\documentclass[11pt,reqno]{amsart}

\numberwithin{equation}{section}
\usepackage{times}
\usepackage{amsmath,amsfonts,amstext,amssymb,amsbsy,amsopn,amsthm,mathrsfs}
\usepackage{dsfont}
\usepackage{graphicx}   
\usepackage{hyperref}
\usepackage{accents}
\usepackage{enumerate}
\usepackage{xcolor}
\usepackage{verbatim}
\usepackage{esint}

\usepackage[normalem]{ulem}
\usepackage{cancel}

\newcommand{\RR}{\mathds{R}}

\newcommand{\cI}{\mathcal{I}}

\newcommand{\cR}{\mathcal{R}}

\newtheorem{theorem}[equation]{Theorem}
\newtheorem*{theorem*}{Theorem}

\newtheorem{proposition}[equation]{Proposition}
\newtheorem{lemma}[equation]{Lemma}
\newtheorem{corollary}[equation]{Corollary}
\theoremstyle{definition}
\newtheorem{definition}[equation]{Definition}

\theoremstyle{remark}
\newtheorem{remark}[equation]{Remark}
\theoremstyle{remark}

\theoremstyle{remark}

\theoremstyle{remark}
\theoremstyle{remark}

\begin{document}


\title[]{Optimal regularity for the arrival time equation}

\author{Yiqi Huang, Jingze Zhu}
\address[Yiqi Huang]{Department of Mathematics, MIT, 77 Massachusetts Avenue, Cambridge, MA 02139-4307, USA}
 \email{yiqih777@mit.edu}
 \address[Jingze Zhu]{Institute of Geometry and Physics, USTC, No. 99 Xiupu Road, Shanghai 201315, China}
 \email{zhujz3@ustc.edu.cn}

\date{\today}
\begin{abstract}
  For a convex mean curvature flow in $\RR^{n+1}$, the corresponding arrival time function solves a degenerate elliptic equation that becomes singular at extinction. We prove that it is $C^{24,\alpha}$ for any $0<\alpha<1$ in the plane, with a logarithmic modulus for its twenty-fourth derivatives. We prove that this bound is sharp combining regularity obstruction in our companion paper \cite{huangzhu}. We obtain asymptotic expansions of the rescaled flow to any prescribed order, including derivative estimates for the remainder. Combined with the work of \v{S}e\v{s}um for $n\ge 2$, this also gives the sharp $C^{2,2/n}$ regularity for convex arrival time functions in higher dimensions. 
\end{abstract}
\maketitle

\setcounter{tocdepth}{1}
\tableofcontents

\section{Introduction}

Let $M_t$ be a smooth, closed and mean-convex hypersurface moving by mean curvature in $\RR^{n+1}$. Let $\Omega$ be the region bounded by $M_0$. The \emph{arrival time} $U$ assigns to each point inside $\Omega$ the time at which the flow reaches it. Thus away from its critical set, its level sets are the hypersurface $M_t$. The equation for these level sets is
\begin{equation}\label{e:intro level set}
\begin{cases}
\Delta U- \frac{D^2U(DU,DU)}{|DU|^2} = -1. & \text{in } \Omega,\\
U = 0 & \text{on } \partial\Omega,
\end{cases}
\end{equation}
The leading coefficient matrix in \eqref{e:intro level set} has a zero eigenvalue in the gradient direction and hence the equation is highly degenerate. Moreover, the equation is undefined at the critical points. The level-set formulation gives a Lipschitz viscosity solution through singularities of mean curvature flow \cite{es1,cgg}, but it does not by itself determine how regular $U$ is beyond Lipschitz; see also \cite{os}.

Equation \eqref{e:intro level set} also lies at the endpoint of a family of degenerate elliptic equations: the $p$-Laplacian. Wherever $DU \neq 0$, 
\begin{equation}\label{e:intro p laplacian}
    |DU|^{2-p} \text{div} (|DU|^{p-2}DU) =\Delta U +(p-2)\frac{D^2U(DU,DU)}{|DU|^2}.
\end{equation}
As $p\downarrow 1$, the right hand side becomes the operator in \eqref{e:intro level set} \cite{e}. The same equation occurs in the deterministic differential games of Kohn and Serfaty: for a convex planar domain, the arrival time arises as a limiting exit-time value function \cite{ks}; see also \cite{sp,koh,ks09,gl}. At the critical point, it is the geometry of the shrinking flow that provides regularity beyond what the formal operator suggests. At the other endpoint $p \to \infty$, At the other endpoint $p\to\infty$, Savin and Evans-Savin proved $C^1$ and $C^{1,\alpha}$ regularity for infinity-harmonic functions in the plane \cite{sa,evsa}. Xu recently obtained the sharp $C^{1,1/3}$ bound through inverse mean curvature flow clusters \cite{xu}.

Back to the arrival time equation, the regularity strongly depends on the flow and the dimension. For general mean-convex flows, Colding and Minicozzi \cite{cm16a,cm18} proved that arrival time is twice differentiable with bounded second derivatives and hence $C^{1,1}$. Generally it is not $C^2$ and the continuity of the Hessian requires additional structure. Their results build on the uniqueness of tangent flows and the structure of the singular set \cite{cm15,cm16b}. 

For convex flows, Huisken \cite{hui93} proved $C^2$ regularity at the extinction point. When $n\ge 2$, \v{S}e\v{s}um found convex examples that are not $C^3$ \cite{se} (see also \cite{str}) and Sun-Xue proved that $C^2$ is the generic regularity \cite{sx24}. For convex curves, Kohn and Serfaty proved $C^3$ regularity in the plane \cite{ks}. However, it remains largely open how regular the planar convex arrival time functions are. 

We prove that the planar arrival time is $C^{24,\alpha}$ for any $0<\alpha<1$. Combining with the regularity obstruction result in our companion paper \cite{huangzhu} proves that this regularity is \emph{optimal}. This identifies the regularity threshold far beyond the earlier $C^3$ bound. To our knowledge, this is the first sharp regularity estimate of such high order for a degenerate elliptic equation.  

\begin{theorem}[Planar regularity]\label{t:intro planar}
    Suppose $U$ be a solution to \eqref{e:intro level set} in some convex $\Omega \subset \RR^2$. Then $U$ is $C^{24,\alpha}$ for every $0<\alpha<1$. More precisely, for any compact $K \subset \Omega$ there is a constant $C_K$ such that
\begin{equation}\label{e:intro modulus}
      |D^{24}U(x)-D^{24}U(y)|
      \leq C_K|x-y| \big(1+|\log|x-y||\big),  \qquad  x,y \in K, x\neq y.
    \end{equation}
    Moreover, there exists some arrival time function $U$ which is not $C^{24,1}$ at some $x_0 \in \Omega$.
\end{theorem}

Previous works mainly identified the leading-order terms of the arrival time. In the plane these terms are polynomials, so they do not reveal the first obstruction to regularity. Therefore, to determine the optimal regularity, we study higher-order expansions. After translating $x_0$ to $0$, we can prove the higher-order expansion near $0$ up to any prescribed order. For the regularity question, we state the expansion only through degree 26; the logarithmic term at degree 25 already gives the obstruction
explicitly. 

\begin{theorem}\label{t:intro planar expansion}
    Suppose $U$ be a solution to \eqref{e:intro level set} in some convex $\Omega \subset \RR^2$. Then there are homogeneous polynomials $P_d$ of degree $d$, $4\le d \le 26$, and a homogeneous polynomial $P_{25}^{\log}$ of
degree $25$ such that
\begin{equation}\label{e:intro expansion}
    T- U(x) = \frac{|x|^2}{2} + \sum_{d=4}^{26} P_d(x) + P_{25}^{\log}(x) \log|x| + \cR(x), \qquad  |D^\beta\cR(x)|\le C_{\beta,\varepsilon} |x|^{27-\varepsilon-|\beta|}
\end{equation}
for every multi-index $\beta$ and $\varepsilon>0$.
\end{theorem}

We have derivative estimates for the remainder term, which is essential for the assertion about $D^{24}U$. If $P_{25}^{\log}\equiv 0$, it also gives $U\in C^{26,\alpha}$ for every $0<\alpha<1$. If $P_{25}^{\log} \not\equiv 0$, twenty-four derivatives of the logarithmic term behave like $r\log r$. This gives precisely the modulus in \eqref{e:intro modulus} and precludes a Lipschitz twenty-fourth derivative. The examples whose $P_{25}^{\log} \not\equiv 0$ are constructed in our companion paper \cite[Theorem 1.3]{huangzhu}, which requires analysis of the regularity obstruction and prescribing higher-order asymptotic expansions.

We explain where the high-order expansion comes from. By the convergence theorems of Gage-Hamilton and Huisken \cite{gh,huisken},  the rescaled convex flow converges to a round sphere $\mathbb{S}^n(\sqrt{2n})$ and the flow can be written as a graph function $v$ over $\mathbb{S}^n(\sqrt{2n})$ satisfying some parabolic equation on $\mathbb{S}^n(\sqrt{2n})$. We expand this graph to any prescribed finite order. 

\begin{theorem}[Asymptotics of the rescaled flow]\label{t:intro graph}
Let $n\ge 1$ and let $v(\cdot, \tau)$ be the graph function of a rescaled mean curvature flow in $\RR^{n+1}$ over $\mathbb{S}^n(\sqrt{2n})$ with $v(\cdot, \tau) \rightarrow 0$ in $C^{\infty}$ as $\tau \rightarrow \infty$. Set
$$
\lambda_k= \frac{k(k+n-1)}{2n}- 1 \quad (k\ge 2), \qquad \mathscr S=\{\lambda_{k_1}+ \cdots +\lambda_{k_q}: q\ge 1,\ k_i\ge 2\}.
$$

There is a unique family $\{F_\mu\}_{\mu\in\mathscr S}$, each $F_\mu(\tau)$ a polynomial in $\tau$ with smooth coefficients on $\mathbb{S}^n(\sqrt{2n})$, such that for all integers $m,\ell,a \ge 0 $, $0<\varepsilon<1$ and sufficiently large $\tau$
\begin{equation}\label{e:intro graph expansion}
\Big\|\partial_{\tau}^a \big(  v-\sum_{\substack{\mu \in\mathscr S\\ \mu \le \frac{m}{2n}}} e^{-\mu\tau}F_\mu(\tau) \big)  \Big\|_{C^\ell(\Sigma)} \le C_{m,\ell,a,\varepsilon}e^{-\frac{(m+1-\varepsilon)\tau}{2n}}.
\end{equation}
\end{theorem}

One could consult Theorem \ref{Thm;asymptotic;2} for more information about the functions $F_{\mu}$. The distinction between the derivative estimates \eqref{e:intro graph expansion} and a function-level asymptotic formula matters. Earlier leading expansions identify possible behavior of $U$ near its critical point, but a remainder that is merely $o(|x|^d)$ cannot be differentiated to obtain regularity of $D^dU$ (see for example \cite{se,str}). We prove the derivative estimates in \eqref{e:intro graph expansion} and transfer them to spatial derivatives after inverting the radial relation between $v$ and $U$. The inversion is finite at each prescribed order, so the same method also computes subsequent arrival-time terms, with powers and possible logarithms.

Using those derivative estimates, we also prove the sharp regularity for convex arrival time functions in higher dimensions: all such arrival time functions are $C^{2,2/n}$. We note that the exponent was identified in the convergence analysis of \v{S}e\v{s}um \cite{se} and the leading order expansion was obtained there. Combining our derivative estimates of the remainder gives the sharp H\"older regularity. 

\begin{theorem}[Higher-dimensional regularity]\label{t:intro higher}
    Let $U$ be a solution to \eqref{e:intro level set} in some convex $\Omega \subset \RR^{n+1}$ for $n\ge 2$.
    Then $U$ is $C^{2,2/n}$.
\end{theorem}

In view of the examples constructed by \v{S}e\v{s}um \cite{se}, this regularity is optimal (see also \cite{se, huangzhu}).

\subsection*{Disclosure of AI tools} The authors use ChatGPT 5.6 Pro for assistance in identifying mathematical and linguistic typos. 

\subsection*{Acknowledgements}
The authors are grateful to Prof. Toby Colding for his continual support and encouragement. The authors would like to thank Prof. Natasa \v{S}e\v{s}um, Prof. Connor Mooney and Zhihan Wang for some helpful conversations. Yiqi Huang is supported by a Simons Dissertation Fellowship.

\section{Preliminaries and setup}

\subsection{Linear operator $L$}
Let $\Sigma = \mathbb{S}^n(\sqrt{2n})$ and let $L = \Delta_{\Sigma} + 1$ be the linearized shrinker operator on the spherical shrinker $\Sigma$. Then we let $\lambda_k$ be the $k$-th eigenvalue of $-L$. We have the explicit formula 
		\begin{align}\label{expression;lambda-k;1}
			\lambda_k = \frac{k(k+n-1)}{2n}-1
		\end{align}

Define $E_k$ to be the eigenspace of $-L$ with eigenvalue $\lambda_k$.

\begin{definition}
	 We first define projections and Sobolev norms  
	 \begin{itemize}
	\item Define the $L^2$ inner product as:
	\begin{align*}
		\langle u,v\rangle_{L^2} := \int_{\Sigma} uv 
	\end{align*}
	\item  Let $\{\phi_{i,j}\}_{1\leq j \leq d_i}$ be an $L^2$ orthonormal basis of $E_i$. Therefore $\{\phi_{i,j}\}_{i\geq 0, 1\leq j \leq d_i}$ is an orthonormal basis of $L^2(\Sigma)$. Then for any $v\in L^2(\Sigma)$ one can decompose $ v = \sum_{i\geq 0, 1\leq j \leq d_i} a_{i,j}\phi_{i,j}$. Clearly $||v||_{L^2}^2 = \sum_{i\geq 0, 1\leq j \leq d_i}a_{i,j}^2$. We define
		\begin{align*}
			||v||_{H^r}^2 = \sum_{\substack{i\geq 0\\ 1\leq j\leq d_i}} a_{i,j}^2 |\lambda_i|^{r} 
		\end{align*}
		Since $\ker(L)=\{0\}$  on $\Sigma$, the above definition is equivalent to the usual $H^r$ norm. Moreover,
		\begin{align*}
			||v||_{H^r}^2 = \langle(-L)^r v, v\rangle
		\end{align*}
		when $v \perp E_0 \oplus E_1 $ 
	\item Define $\Pi_{\lambda_k}$ to be the $L^2$ orthogonal projection onto $E_k$ and define $\Pi_{\mu} = 0$ if $\mu \neq \lambda_k$ for any $k\geq0$.  For all $\mu \geq 0$, define $\Pi_{\geq \mu}$ to be the $L^2$ orthogonal projection onto $\displaystyle \bigoplus_{\lambda_k \geq \mu}E_k$. Define similarly for $\Pi_{>\mu}$, $\Pi_{\leq\mu}$, $\Pi_{<\mu}$.  
\end{itemize}
\end{definition}

\begin{proposition}\label{Prop;L-invertible}
For any $\mu\in\RR$ and $\ell\geq2$, the following operator is invertible:
\begin{equation*}
    L+\mu : H^\ell \cap(1-\Pi_\mu)L^2
  \longrightarrow H^{\ell-2} \cap(1-\Pi_\mu) L^2.
\end{equation*}
Moreover, $ \|(L+\mu)^{-1}f \|_{H^\ell} \leq C_{\ell,\mu} \| f \|_{H^{\ell-2}}$.\end{proposition}

\begin{proof}
On $E_k$, $L+\mu$ acts by multiplication by $\mu-\lambda_k$. Hence, on the orthogonal complement of its kernel,
\begin{equation*}
    (L+\mu)^{-1}f = \sum_{\lambda_k\neq \mu} \frac{\Pi_{\lambda_k} f}{\mu-\lambda_k}.
\end{equation*}

The estimate follows from the definition of the Sobolev norms and
$\sup_{\lambda_k \neq \mu}|\lambda_k| / |\mu-\lambda_k|< \infty$.
\end{proof}

We will now define several sets:
\begin{definition}
	$\mathscr{S}= \{\lambda_{l_1} + \cdots + \lambda_{l_m}: m\geq 1, l_i \geq 2\}$. 
	In view of \eqref{expression;lambda-k;1},  $\mathscr{S}\subset \frac{1}{2n}\mathbb{Z}$.
\end{definition}

\begin{definition}\label{Def;function-class;1}
	Let $\mathscr{F}$ be the collection of time dependent functions on $\Sigma$ in the form
\begin{align}
	F = F(\tau) = F(\tau, x) = \sum_{j=0}^{d}\tau^{j}f_j(x),\qquad f_0,\cdots, f_d \in C^{\infty}(\Sigma), \quad f_d\neq0 
\end{align}
The degree of such $F$ is defined to be $d$. If $F \equiv 0$, then we define its degree to be $-\infty$. 

Define $\mathscr{E}$ to be the collection of functions on $\Sigma$ in the form:
\begin{align}
	G = \sum_{\mu\in \mathscr{S}} e^{-\mu \tau}G_{\mu}(\tau), \qquad G_{\mu}(\tau) \in \mathscr{F}
\end{align}
where $G_{\mu}\not\equiv 0$ only for \textbf{finitely many} $\mu$ and $G_{\mu} \in \mathscr{F}$ for all $\mu\in\mathscr{S}$. 

For each $G \in \mathscr{E}$ written in the above form, define  $[G]_{\mu} = G_{\mu}$ if $\mu \in  \mathscr{S}$ and define $[G]_{\mu} = 0$ otherwise. 
\end{definition}

We record some basic facts about $\mathscr{E}$ and $\mathscr{F}$:
\begin{proposition}\label{Prop;E-F-derivative;2}  
\begin{itemize}
	\item $\mathscr{E}$ and $\mathscr{F}$ are preserved under any time derivatives. They are also preserved by addition, multiplication, spatial derivatives and their compositions that preserve $C^{\infty}(\Sigma)$. (e.g., the operator $L$, the contraction $|\nabla \cdot|^2$ . )
	\item Any element in $\mathscr{E}$ has a unique representation. Namely, if an element in $\mathscr{E}$ can be written in two forms:
 	\begin{align}
		\sum_{\mu\in \mathscr{S} } e^{-\mu \tau}G_{\mu}(\tau) = \sum_{\mu\in \mathscr{S} } e^{-\mu \tau}G'_{\mu}(\tau)
	\end{align} 
	  then $G_{\mu} = G_{\mu}'$ for all $ \mu $. 
\end{itemize}
\end{proposition}
\begin{proof}
	The first statement follows from direct computation. Multiplication keeps the exponential rate $\mu$ in $\mathscr{S}$, since $\mathscr{S}$ is closed under addition. The second statement follows from comparing asymptotics.
\end{proof}

\begin{proposition}\label{Prop;unique-linear}
Consider the homogeneous equation
\begin{align}
	\partial_{\tau} F - (L+\mu)F = 0
\end{align}
If $\mu \neq \lambda_k$ for any $k\geq 0$, then the only solution in $\mathscr{F}$ is 0. 
If $\mu = \lambda_k$ for some $k\geq 0$, then all the solutions in $\mathscr{F}$ are precisely $E_k$.
\end{proposition}
\begin{proof}

\noindent\textbf{Case 1:}  $\mu \neq \lambda_k$ for any $k\geq 0$. 

Suppose that there is a nonzero solution $F$, we aim to derive a contradiction. By definition we can write $F = \sum_{j=0}^{d}\tau^j f_j(x)$, where $d\geq 0$, $f_0,\cdots, f_d\in C^{\infty}$ and $f_d \not\equiv 0$.  We compute: 
\begin{align}\label{Eq;F;1}
	0= \partial_{\tau}^d(\partial_{\tau} F - (L+\mu)F) = \partial_{\tau}^{d+1}F - (L+\mu)\partial_{\tau}^d F = -d! (L+\mu)f_d
\end{align}
Since $f_d \neq 0$, this forces $\mu$ to be an eigenvalue $\lambda_k$ for some $k\geq 0$, a contradiction.

\noindent\textbf{Case 2:} $\mu =\lambda_k$ for some $k\geq 0$. We take any nonzero solution $F$ and again write $F = \sum_{j=0}^{d}\tau^j f_j(x)$, where $d\geq 0$, $f_0,\cdots, f_d\in C^{\infty}$ and $f_d \not\equiv 0$. Then \eqref{Eq;F;1} still applies and we get $f_d\in E_k$. It remains to prove that $d = 0$. If this is not the case, then we have
\begin{align}\label{Eq;F;2}
	0= \partial_{\tau}^{d-1}(\partial_{\tau} F - (L+\mu)F) = \partial_{\tau}^dF - (L+\mu)\partial_{\tau}^{d-1} F = d!f_d - (L+\lambda_k) (d-1)! f_{d-1}
\end{align}
This means $(L+\lambda_k)f_{d-1}  = df_d \in E_k $. However, this is a contradiction because $E_k$ is the cokernel of $L+\lambda_k$. 

\end{proof}

\subsection{Nonlinear error $\mathcal{N}$}

Suppose that a rescaled mean curvature flow $\bar{M}_{\tau}$ can be written as a graph over $\Sigma$ with graph function $v$. Then $v$ satisfies the evolution equation
\begin{align}\label{Eq;RMCF;1}
	\partial_{\tau} v- Lv = \mathcal{N} (v,\nabla v, \nabla^2 v)
\end{align} 	
We may abbreviate $\mathcal{N} (v,\nabla v, \nabla^2 v)$ as $\mathcal{N} (v)$ when no confusion arises. 

$\mathcal{N}$ can be viewed as function on $C^{2}(\Sigma) \otimes T\Sigma \otimes Sym(T^2\Sigma)$. Moreover $\mathcal{N}$ is quadratic, meaning that $\mathcal{N}(0,0,0) = 0$ and $D\mathcal{N}|_{(0,0,0)} = 0$. In addition, we can find $\varepsilon_1 \in (0,1)$ such that 
\begin{align}\label{Def;varepsilon-1;1}
	\mathcal{N}(z,p,X) \text{ is analytic when } |z|, |p|, |X| < 2\varepsilon_1
\end{align}
where the norm is taken w.r.t to Riemannian metric on $\Sigma$.

For all $h\in C^2$, define 
\begin{align}
	J(h) := (h,\nabla h, \nabla^2 h) \in  T(C^{2}(\Sigma) \otimes T\Sigma \otimes Sym(T^2\Sigma))
\end{align} 
We denote the differential of  $\mathcal{N}$ at $(w,\nabla w, \nabla^2 w)$ in the direction of $J(h)$ by
\begin{align}
	D\mathcal{N}|_{(w,\nabla w, \nabla^2 w)}[J(h)] = D_z \mathcal{N}  [h] + D_p \mathcal{N}  [\nabla h] + D_X \mathcal{N}  [\nabla^2h ].
\end{align} 
Higher derivatives in the directions of $J(h_1), \cdots, J(h_q)$ are denoted similarly:
\begin{align}
	&D^q\mathcal{N}\big|_{(w,\nabla w, \nabla^2 w)}[J(h_1), \cdots ,J(h_q)] \nonumber \\
	=&\sum_{\substack{a+b+c = q\\ a,b,c \geq 0}} \frac{1}{a!b!c!}\sum_{\substack{\{i_1,\cdots, i_a, \\   j_1,\cdots, j_b, \\ k_1,\cdots, k_c\} \\
=\{1,\cdots, q\}}} D_z^a D_p^b D_X^c \mathcal{N}\big|_{(w,\nabla w, \nabla^2 w)}  \Big[ h_{i_1},\cdots, h_{i_a}, \nabla h_{j_1},\cdots, \nabla h_{j_b},\nabla^2 h_{k_1},\cdots, \nabla^2 h_{k_c}\Big]
\end{align} 

In particular, if we take the base point to be $(0,0,0)$ and take all the directions to be the same, then we get the Taylor polynomial of $\mathcal{N}$: 

\begin{definition}\label{Def;N;1}
Define the $q$-th Taylor polynomial of $\mathcal{N}(v,\nabla v, \nabla^2 v)$ as
\begin{align}
	\mathcal{N}_{q}(v) := \frac{1}{q!}\frac{d^q}{d s^q}\Big|_{s=0}\mathcal{N}(sv, s\nabla v, s\nabla^2 v)
\end{align}
This is equivalent to 
\begin{align}
	\mathcal{N}_{q}(v) =& \frac{1}{q!} D^q\mathcal{N}|_{(0,0,0)}[\underbrace{J(v),\cdots, J(v)}_{q}] \nonumber \\
	 =  \sum_{a+b+c = q} &\frac{1}{a!b!c!}D_z^a D_p^b D_X^c \mathcal{N}\big|_{(0,0,0)}  \Big[\underbrace{v,\cdots,v}_{a}, \underbrace{\nabla v,\cdots, \nabla v}_{b}, \underbrace{\nabla^2 v,\cdots, \nabla^2 v}_{c}\Big]
\end{align}
Define
\begin{align}
	\mathcal{N}_{\leq m}(v) = \sum_{q=0}^{m} \mathcal{N}_q(v), \qquad  \mathcal{N}_{>m}(v) =\mathcal{N}(v, \nabla v, \nabla^2 v) - \mathcal{N}_{\leq m}(v)
\end{align}
Note that $\mathcal{N}_0 = 0 $ and $ \mathcal{N}_1 = 0$ by quadraticity of $\mathcal{N}$. 
\end{definition}

The above definition suggests the following:
\begin{lemma}\label{Lem;N-m;1}
For all $m\geq 0$, 	$\mathcal{N}_{\leq m}$ preserves $\mathscr{E}$.
\end{lemma}
\begin{proof}
	The result is trivial when $m=0, 1$, so we fix $m\geq 2$ and fix any element $G \in \mathscr{E}$.
Recall that for all $q\leq m$: 
\begin{align}
		\mathcal{N}_{q}(G)  =  \sum_{a+b+c = q} &\frac{1}{a!b!c!}D_z^a D_p^b D_X^c \mathcal{N}\big|_{(0,0,0)}  \Big[\underbrace{G,\cdots,G}_{a}, \underbrace{\nabla G,\cdots, \nabla G}_{b}, \underbrace{\nabla^2 G,\cdots, \nabla^2 G}_{c}\Big]
\end{align}
Therefore $\mathcal{N}_{q}(G)$ is a polynomial in $G, \nabla G$ and $\nabla^2 G$. Then the assertion follows from Proposition \ref{Prop;E-F-derivative;2}.

\end{proof}

We prove some useful facts about the functional $\mathcal{N}$:
\begin{lemma}\label{Lem;N-diff;1}
	Let $a,\ell \geq 0$.  Suppose that $v_1,v_2$ satisfies  $||v_1||_{C^{\ell+2}} \leq \varepsilon_1, ||v_2||_{C^{\ell+2}} \leq \varepsilon_1 $. In addition, suppose that there is   $B>0$ such that $||\partial_{\tau}^j v_i ||_{C^{\ell+2}} \leq B$ for all $i=1,2$, $j=0,1,...,a$ . Then there is a constant $C_{a,\ell}$ such that the difference estimate holds:
\begin{align}
	\big\Vert\partial_{\tau}^a \Big(\mathcal{N}(v_1) - \mathcal{N}(v_2)\Big) \big\Vert_{C^\ell} \leq   C_{a,\ell}\cdot (1+B)^a\sum_{\substack{j,k\geq 0 \\ j+k \leq a}} \Big(\big\Vert  \partial_{\tau}^j v_1 \big\Vert_{C^{\ell+2}} + \big\Vert  \partial_{\tau}^j v_2 \big\Vert_{C^{\ell+2}} \Big) \cdot \big\Vert \partial_{\tau}^k (v_1 - v_2) \big\Vert_{C^{\ell+2}}
\end{align}
\end{lemma}
\begin{proof}

For any $w$ with $||w||_{C^{2}} \leq \varepsilon_1$ and $h_1,...h_q \in C^2$ we have the general formula:
\begin{align}\label{formula;diff-DN;2}
	\partial  D^q\mathcal{N}_{(w,\nabla w, \nabla^2 w)}& [J(h_1),\cdots, J(h_q)] =   D^{q+1}\mathcal{N}_{(w,\nabla w, \nabla^2 w)} [\partial J(w), J(h_1),\cdots, J(h_q)]  \nonumber\\
	+& D^q\mathcal{N}_{(w,\nabla w, \nabla^2 w)} [ \partial J( h_1),\cdots, J(h_q)]  +\cdots +D^q\mathcal{N}_{(w,\nabla w, \nabla^2 w)} [J(h_1),\cdots,  \partial J( h_q)] 
\end{align}
where $\partial$ is fixed but can be chosen from $\nabla_i, \nabla$ or $\partial_{\tau}$.
We consider the difference 
\begin{align}
	\mathcal{N}(v_1) - \mathcal{N}(v_2) = \int_0^1 D\mathcal{N}|_{(v^{s},\nabla v^{s}, \nabla^2 v^{s})}[J(v_1-v_2)]ds
\end{align}
where $v^{s} = v_2 + s(v_1-v_2)$. 
Since $D\mathcal{N}|_{(0,0,0)} = 0$, we further get
\begin{align}
	\mathcal{N}(v_1) - \mathcal{N}(v_2) =  \int_0^1\int_0^1 D^2\mathcal{N}|_{(tv^{s},\nabla tv^{s}, t\nabla^2 v^{s})}[J(v^s), J(v_1-v_2)]dtds
\end{align}
Iterating \eqref{formula;diff-DN;2} $a+\ell$ times we get
\begin{align}
	& \partial^{a+\ell} (\mathcal{N}(v_1) - \mathcal{N}(v_2))	=   \sum_{j=0}^{a+\ell}\sum_{\substack{k_1,\cdots, k_j\geq 1 \\ k_{j+1}, k_{j+2} \geq 0 \\ k_1+\cdots + k_{j+2} = a+\ell}} \nonumber  C_{k_1,\cdots, k_{j+2}}\cdot \\
	&\int_0^1\int_0^1 D^{2+j}\mathcal{N}|_{(tv^{s},\nabla tv^{s}, t\nabla^2 v^{s})} \Big[\partial^{k_1}J( tv^s), ..., \partial^{k_j}J(tv^s), \partial^{k_{j+1}}J(v^s), \partial^{k_{j+2}}J(v_1-v_2)\Big]dtds
\end{align}
where each $\partial$ can be chosen from $\nabla_i, \nabla$ or $\partial_{\tau}$. But they must match on both sides. Now apply this with $\partial^{a+\ell} = \nabla^\ell \partial_{\tau}^a$.  
Since $||tv^{s}|| \leq \max\{||v_1||_{C^{\ell+2}}, ||v_2||_{C^{\ell+2}} \} \leq \varepsilon_1$, by analyticity we have $\Big|D^{2+j}\mathcal{N}|_{(tv^{s},\nabla tv^{s}, t\nabla^2 v^{s})}\Big| \leq C_{a,\ell}$ when $j\leq a+\ell$. Moreover $||J(h)||_{C^0} \le C ||h||_{C^2}$. Therefore,
\begin{align}
	\big\Vert\partial_{\tau}^a \Big(\mathcal{N}(v_1) - \mathcal{N}(v_2)\Big)\big\Vert_{C^\ell} \leq  & C_{a,\ell} \sum_{j=0}^{a+\ell} \sum_{\substack{   k_1+\cdots + k_{j+2} = a}}   \prod_{i=1}^{j+1} \Big(\big\Vert\partial_{\tau}^{k_{i}} v_1 \big\Vert_{C^{\ell+2}}+\big\Vert\partial_{\tau}^{k_{i}} v_2 \big\Vert_{C^{\ell+2}} \Big)  \cdot \big\Vert\partial_{\tau}^{k_{j+2}} (v_1 - v_2)\big\Vert_{C^{\ell+2}}  \nonumber \\
	\leq & C_{a,\ell}(1+B)^a \sum_{j+k \leq a}\Big(\big\Vert\partial_{\tau}^j v_1 \big\Vert_{C^{\ell+2}}+\big\Vert\partial_{\tau}^j v_2 \big\Vert_{C^{\ell+2}} \Big)  \cdot \big\Vert\partial_{\tau}^k (v_1 - v_2)\big\Vert_{C^{\ell+2}} 
\end{align}
Here all  spatial derivatives are incorporated into $C^{\ell+2}$-norms.
The factor $(1+B)$ only receives power $a$ because there are at most $a$ nonzero elements among $k_1,\cdots, k_{j+1}$. 

\end{proof}

\begin{lemma}\label{Lem;N-tail;1}
	Let $a,\ell \geq 0$, $m \geq 2$ and $||v||_{C^{\ell+2}} \leq \varepsilon_1$. Suppose that there is $B>0$ such that $||\partial_{\tau}^j v  ||_{C^{\ell+2}} \leq B$ for all $0\leq j \leq a$, then there is a constant $C_{m, a,\ell}$ such that
	\begin{align}
		\big\Vert\partial_{\tau}^a\mathcal{N}_{>m}(v) \big\Vert_{C^\ell} \leq  C_{m, a,\ell}\cdot (1+B)^a \sum_{a_1+\cdots+a_{m+1} \leq a} \prod_{i=1}^{m+1} \big\Vert \partial_{\tau}^{a_i}v\big\Vert_{C^{\ell+2}}
	\end{align}
\end{lemma}
\begin{proof}
By definition \ref{Def;N;1} and the fact that $\mathcal{N}(0) = 0$,  Taylor remainder formula gives
\begin{align}
	\mathcal{N}_{> m}(v) =& \mathcal{N}(v) - \mathcal{N}(0) - \sum_{q=0}^m \frac{1}{q!}\frac{d^q}{d s^q}\Big|_{s=0}\mathcal{N}(sv, s\nabla v, s\nabla^2 v) \nonumber\\
	=& \frac{1}{m!} \int_0^1 (1-s)^m \frac{d^{m+1}}{ds^{m+1}}\mathcal{N} {(sv, s\nabla v, s\nabla^2 v)} ds
\end{align}
We can then rewrite this as
\begin{align}
	\mathcal{N}_{> m}(v)  	= \frac{1}{m!} \int_0^1 (1-s)^m D^{m+1}\mathcal{N}\big|_{(sv,s\nabla v, s\nabla^2 v)}[\underbrace{J(v), \cdots, J(v)}_{m+1}] ds
\end{align}
Now we iterate \eqref{formula;diff-DN;2} $a+\ell$ times:
\begin{align}
	&\partial^{a+\ell} \mathcal{N}_{> m}(v) = \frac{1}{m!}\sum_{j=0}^{a+\ell}  \sum_{\substack{k_1+\cdots +k_j +\\ a_1+\cdots+ a_{m+1} = a +\ell \\ k_i\geq 1, a_i \geq 0}}  C_{k_1,\cdots, k_j, a_1,\cdots, a_{m+1}}\cdot \nonumber \\
	&\int_0^1 (1-s)^m D^{m+1 + j}\mathcal{N}|_{(sv, s\nabla v, s\nabla^2 v)} \big[\partial^{k_1} J(sv), .., \partial^{k_j} J(sv), \partial^{a_1} J(v), .. , \partial^{a_{m+1}} J(v) \big] ds
\end{align}
where  each $\partial$ can be chosen from $\nabla_i, \nabla$ or $\partial_{\tau}$, but they must match on both sides. 

Since $||sv||_{C^2} \leq \varepsilon_1$ for all $s\in [0,1]$, by the analyticity  of $\mathcal{N}$ in \eqref{Def;varepsilon-1;1}, we have the  bound on the coefficient:
 $\Big|D^{m+1 + j}\mathcal{N}\big|_{(sv, s\nabla v, s\nabla^2 v)}\Big| \leq C_{m,a,\ell}$ for $0\le j \le a+\ell$. Now we take $\nabla^\ell \partial_{\tau}^a$ in place of $\partial^{a+\ell} $ to obtain:
 \begin{align}
 	\Big\Vert\partial_{\tau}^a\mathcal{N}_{> m}(v) \Big\Vert_{C^\ell} & \leq   C_{m,a,\ell}\sum_{j=0}^{a} \sum_{\substack{k_1+\cdots +k_j +\\ a_1+\cdots+ a_{m+1} = a  \\ k_i\geq 1, a_i \geq 0}}  \prod_{i=1}^{j}  \big\Vert\partial_{\tau}^{k_{i}} (sv) \big\Vert_{C^{\ell+2}}  \cdot  \prod_{i=1}^{m+1}\big\Vert\partial_{\tau}^{a_i} v \big\Vert_{C^{\ell+2}}  
 \end{align}
 Here all  spatial derivatives are incorporated into $C^{\ell+2}$-norms. 
In the summation,  there are at most $a - a_1-\cdots a_{m+1}$ nonzero element among $k_1,\cdots, k_{j}$. Therefore,  we rearrange the estimate as
 \begin{align}
 	\Big\Vert\partial_{\tau}^a\mathcal{N}_{> m}(v) \Big\Vert_{C^\ell} 
 	 &\leq  C_{m,a,\ell}\sum_{j=0}^{a} \sum_{\substack{ a_1+\cdots+ a_{m+1} \leq a -j}} B^j   \prod_{i=1}^{m+1}\big\Vert\partial_{\tau}^{a_i} v \big\Vert_{C^{\ell+2}}  \nonumber\\
 	 & \leq C_{m,a,\ell} (1+B)^{a}  \sum_{\substack{ a_1+\cdots+ a_{m+1} \leq a}} \prod_{i=1}^{m+1}\big\Vert\partial_{\tau}^{a_i} v \big\Vert_{C^{\ell+2}}
 \end{align}

\end{proof}

\section{Fine asymptotic expansion}\label{s:fine asymptotic expansion}
Recall that a rescaled mean curvature flow $\bar{M}_{\tau}$ can be written as a graph over $\Sigma = \mathbb{S}^{n}(\sqrt{2n})$, and the graph function satisfies
\begin{align}\label{Eq;RMCF;2}
	\partial_{\tau}v = Lv + \mathcal{N}(v,\nabla v, \nabla^2 v)
\end{align}
We might abbreviate $\mathcal{N}(v,\nabla v, \nabla^2 v) = \mathcal{N}(v)$ when no confusion arises.  

The goal of this section is to prove the following Theorem, which gives the asymptotic expansion of the graph function up to \textbf{any} order:
\begin{theorem}\label{Thm;asymptotic;2}
	Suppose that $v$ solves \eqref{Eq;RMCF;2} and $v \rightarrow 0$ in $C^{\infty}$ as $\tau \rightarrow \infty$. Then there exists a unique family $\{F_{\mu}\}_{\mu \in \mathscr{S}}\subset \mathscr{F}$ such that for all integers $m, \ell\geq 0$ and $\varepsilon \in (0,1)$,   
\begin{align}
	\big\Vert v - \sum_{\substack{\mu  \in \mathscr{S} \\ \mu \leq \frac{m}{2n}}}e^{-\mu\tau}F_{\mu}(\tau) \big\Vert_{C^\ell(\Sigma)} = O(e^{-\frac{m+1-\varepsilon}{2n}\tau})
\end{align}
Moreover, $F_{\mu}$ satisfies the following formula:
\begin{align}\label{Eq;F;3}
	F_{\frac{m}{2n}}' - \Big(L+\frac{m}{2n}\Big) F_{\frac{m}{2n}} = \Biggr[\mathcal{N}_{\leq m} \Biggr(\sum_{\substack{\mu  \in \mathscr{S} \\ \mu \leq \frac{m-1}{2n}}}e^{-\mu\tau}F_{\mu}(\tau)\Biggr) \Biggr]_{\frac{m}{2n}}
\end{align}
for all integer $m\geq 2$, where we make the convention that $F_{\frac{m}{2n}} = 0$ if $\frac{m}{2n}\not\in \mathscr{S}$.
\end{theorem}

Our strategy is to treat the nonlinear term $\mathcal{N}$ as if it is a non-autonomous term and see how $v$ responds. The key is the following Lemma:
\begin{lemma}\label{Lem;linear;1}
	Let $m\geq 0$ be an integer and $\varepsilon \in (0,1)$. Given a smooth function $v(x,\tau)$ satisfying 
	\begin{align}\label{Eq;v;3}
		\partial_{\tau}v - Lv = G
	\end{align} 
	with $v \rightarrow 0$ in $C^{\infty}$ as $\tau \rightarrow \infty$. Suppose that there exist $\{G_{\mu}\}_{\mu \in \mathscr{S}\cap (0, \frac{m}{2n}]} \subset \mathscr{F}$ such that 
	\begin{align}
		\big\Vert G - \sum_{\mu \in \mathscr{S}\cap (0, \frac{m}{2n}]}e^{-\mu\tau} G_{\mu} \big\Vert_{C^\ell} = O(e^{-\frac{m+1-\varepsilon}{2n}\tau})
	\end{align}
	for all $\ell \geq 0$. Then there exists a unique collection $\{F_{\mu}\}_{\mu \in \mathscr{S}\cap (0, \frac{m}{2n}]} \subset \mathscr{F}$ such that
	\begin{align}\label{estimate;v;2}
		\big\Vert v - \sum_{\mu \in \mathscr{S}\cap (0, \frac{m}{2n}]}e^{-\mu\tau} F_{\mu} \big\Vert_{C^\ell} = O(e^{-\frac{m+1-\varepsilon}{2n}\tau})
	\end{align}
	for all $\ell \geq 0$. Moreover, for each $\mu \in \mathscr{S}\cap (0, \frac{m}{2n}]$,  $F_{\mu}$ solves the equation 
	\begin{align}\label{Eq;F-G;2}
		\partial_{\tau}F_{\mu} - (L+\mu) F_{\mu} = G_{\mu}
	\end{align}
\end{lemma}

We first present a technical proposition, the purpose of which will be clear when we prove Lemma \ref{Lem;linear;1}.
\begin{proposition}\label{Prop;identity;2}
	Suppose that  $\mu \in \mathscr{S}$, $G_{\mu}\in \mathscr{F}$ and the degree of $G_{\mu}$ is $d_{\mu}$. Let $\Pi_{\mu} = 0$ if $\mu \neq \lambda_k$ for any $k\geq 0$. Then we have the following identity
\begin{align}
	 \int_0^{\tau}e^{L(\tau-s)}e^{-\mu s}	G_{\mu}(s)ds  =& e^{-\mu \tau}\Biggr(-\sum_{j=0}^{d_{\mu}} (L+\mu)^{-j-1}\partial_{\tau}^j (1-\Pi_{\mu})G_{\mu}(\tau) \nonumber  + \int_0^{\tau} \Pi_{\mu} G_{\mu}(s)ds\Biggr) \nonumber \\
	  &+     e^{L\tau}\sum_{j=0}^{d_{\mu}} (L+\mu)^{-j-1} \frac{\partial^j}{\partial s^{j}}\Big|_{s=0} (1-\Pi_{\mu})G_{\mu}
\end{align}
In particular, for any $k\geq 0$,  there exist $\bar{G}_1$ and $\bar{G}_2$ in $\mathscr{F}$ such that
\begin{align}
	\int_0^{\tau}e^{L(\tau-s)}e^{-\mu s}\Pi_{\lambda_k}	G_{\mu}(s) ds = e^{-\mu \tau} \bar{G}_1 + e^{-\lambda_k\tau} \bar{G}_2
\end{align}
\end{proposition}
\begin{remark}
    The operator $(L+\mu)^{-1}$ is justified in Proposition \ref{Prop;L-invertible}.
\end{remark}
\begin{proof}Let us prove the first statement.
Note that the RHS is equal to
\begin{align}
	&  e^{-\mu\tau} \sum_{j=0}^{d_{\mu}}\Big( -e^{(L+\mu)(\tau-s)}(L+\mu)^{-j-1}\frac{\partial^j}{\partial s ^{j}}(1-\Pi_{\mu})G_{\mu} (s)\Big|_{s=0}^{\tau}\Big) + e^{-\mu\tau}\int_0^{\tau}\Pi_{\mu}G_{\mu}(s)ds\nonumber\\
	=& e^{-\mu\tau}\int_0^{\tau}\sum_{j=0}^{d_{\mu}} \frac{d}{ds} \Big( -e^{(L+\mu)(\tau-s)}(L+\mu)^{-j-1}\frac{\partial^j}{\partial s ^{j}} (1-\Pi_{\mu}) G_{\mu} (s)\Big)ds  + e^{-\mu\tau}\int_0^{\tau}\Pi_{\mu}G_{\mu}(s)ds\nonumber\\
	=& e^{-\mu\tau}\int_0^{\tau}\sum_{j=0}^{d_{\mu}}  \Big( e^{(L+\mu)(\tau-s)}(L+\mu)^{-j}\frac{\partial^j}{\partial s ^{j}}(1-\Pi_{\mu})G_{\mu} (s) - e^{(L+\mu)(\tau-s)}(L+\mu)^{-j-1}\frac{\partial^{j+1}}{\partial s ^{j+1}}(1-\Pi_{\mu})G_{\mu} (s)\Big)ds \nonumber\\
	&+e^{-\mu\tau}\int_0^{\tau}\Pi_{\mu}G_{\mu}(s)ds \nonumber\\
	=& e^{-\mu\tau}\int_0^{\tau}  \Big( e^{(L+\mu)(\tau-s)} (1-\Pi_{\mu})G_{\mu} (s) - e^{(L+\mu)(\tau-s)}(L+\mu)^{-d_\mu-1}\frac{\partial^{d_{\mu}+1}}{\partial s ^{d_{\mu}+1}}(1-\Pi_{\mu})G_{\mu} (s)\Big)ds \nonumber \\
	 &+ e^{-\mu\tau}\int_0^{\tau}\Pi_{\mu}G_{\mu}(s)ds \nonumber\\
	=& \int_0^{\tau}    e^{L(\tau-s)} e^{-\mu s}(1-\Pi_{\mu})G_{\mu} (s) ds + e^{-\mu\tau}\int_0^{\tau}\Pi_{\mu}G_{\mu}(s)ds
\end{align}
where we used the fact that $\frac{\partial^{d_{\mu}+1}}{\partial s ^{d_{\mu}+1}}(1-\Pi_{\mu})G_{\mu} (s) =0$ to get the last line. Next, we observe that $L+\mu = 0$ on the image of $\Pi_{\mu}$. Therefore, the last term is:  
\begin{align}
	e^{-\mu\tau}\int_0^{\tau}\Pi_{\mu}G_{\mu}(s)ds =e^{-\mu\tau}\int_{0}^{\tau} e^{(L+\mu)(\tau-s)} \Pi_{\mu}G_{\mu}(s)ds =  \int_{0}^{\tau} e^{L(\tau-s)} e^{-\mu s}\Pi_{\mu}G_{\mu}(s)ds 
\end{align}
Putting them together, the assertion follows immediately.

For the second statement, since projection $\Pi_{\lambda_k}$ preserves $\mathscr{F}$ we use $\Pi_{\lambda_k}G_{\mu}$ in place of $G_{\mu}$ to obtain
\begin{align}
	&\int_0^{\tau}e^{L(\tau-s)}	e^{-\mu s}\Pi_{\lambda_k} G_{\mu}(s)ds  \nonumber\\
	 =& e^{-\mu \tau}\Biggr( -\sum_{j=0}^{d_{\mu}} (L+\mu)^{-j-1}\partial_{\tau}^j (1-\Pi_{\mu})\Pi_{\lambda_k} G_{\mu}(\tau)  + \int_0^{\tau} \Pi_{\mu} \Pi_{\lambda_k} G_{\mu}(s)ds\Biggr) \nonumber \\ 
	 &+     e^{L\tau}\sum_{j=0}^{d_{\mu}} (L+\mu)^{-j-1} \frac{\partial^j}{\partial s^{j}}\Big|_{s=0} (1-\Pi_{\mu})\Pi_{\lambda_k} G_{\mu}
\end{align}
Set 
\begin{align}
	\bar{G}_1 =  &-\sum_{j=0}^{d_{\mu}} (L+\mu)^{-j-1}\partial_{\tau}^j (1-\Pi_{\mu})\Pi_{\lambda_k} G_{\mu}(\tau)  + \int_0^{\tau} \Pi_{\mu} \Pi_{\lambda_k} G_{\mu}(s)ds \\
	\bar{G}_2 = &\sum_{j=0}^{d_{\mu}} (L+\mu)^{-j-1} \frac{\partial^j}{\partial s^{j}}\Big|_{s=0} (1-\Pi_{\mu})\Pi_{\lambda_k} G_{\mu}
\end{align} 
They are both elements of $\mathscr{F}$. Moreover, $e^{L\tau} = e^{-\lambda_{k}\tau}$ on the image of $\Pi_{\lambda_k}$, the assertion then follows. 
\end{proof}
\begin{proof}[Proof of Lemma \ref{Lem;linear;1}]
The uniqueness is clear. We only prove the existence. After translating the initial time, we may work on $[0,\infty)$.

First by the solution formula we get
\begin{align}\label{formula;v;11}
	v =& e^{L\tau} v(0) + \int_0^{\tau} e^{L(\tau-s)}G(s)ds \nonumber \\
	=&e^{L\tau} v(0) + \int_0^{\tau} e^{L(\tau-s)}\hat{G}(s)ds + \sum_{\mu \in \mathscr{S}\cap (0, \frac{m}{2n}]} \int_0^{\tau}    e^{L(\tau-s)} e^{-\mu s} G_{\mu} (s) ds
\end{align}
We begin by constructing $F_{\mu}$. 
Let 
\begin{align}\label{estimate;hat-G;1}
	\hat{G} =  G - \sum_{\mu \in \mathscr{S}\cap (0, \frac{m}{2n}]}e^{-\mu\tau} G_{\mu} 
\end{align}
By our assumption, for all $\ell\geq 0$
\begin{align}
	||\hat{G}||_{H^\ell} = O(e^{-\frac{m+1-\varepsilon}{2n}\tau})
\end{align}
Choose $K$ to be the integer such that $\lambda_{K-1} < \frac{m+1-\varepsilon}{2n}<\lambda_K$. 

The Lemma will be proved in 4 steps. The idea is to use $v_m$ to represent $ \sum_{\mu \in \mathscr{S}\cap (0, \frac{m}{2n}]}e^{-\mu\tau} F_{\mu} $, and then show that $||v- v_m||_{C^\ell} =  O(e^{-\frac{m+1-\varepsilon}{2n}\tau})$ for all $\ell$. In step 1, we construct $v_m$ and $\hat{v}$, then show that $v- v_m = \hat{v}$.  In step 2, we find the equation for $v_m$ and $\hat{v}$. In step 3, we derive the decay estimate $\hat{v} = O(e^{-\frac{m+1-\varepsilon}{2n}\tau})$.  In step 4, we show that $v_m= \sum_{\mu \in \mathscr{S}\cap (0, \frac{m}{2n}]}e^{-\mu\tau} F_{\mu} $ and $F_{\mu}$ satisfies the right equation.

\noindent\textbf{Step1:} Decompose $v$ into $v_m$ and $\hat{v}$.

Let us define
\begin{align}\label{formula;v-m;12}
	v_m :=& e^{L\tau}\Pi_{< \lambda_K}v(0) + \int_0^{\infty} e^{L(\tau-s)}\Pi_{< \lambda_K}\hat{G}(s)ds  +\sum_{\mu \in \mathscr{S}\cap (0, \frac{m}{2n}]}\int_0^{\tau}e^{L(\tau-s)} \Pi_{<\lambda_K}\Big(e^{-\mu s}G_{\mu}(s)\Big)ds  \nonumber\\
	&-\sum_{\mu \in \mathscr{S}\cap (0, \frac{m}{2n}]}e^{-\mu\tau} \sum_{j=0}^{d_{\mu}} (L+\mu)^{-j-1}\partial_{\tau}^j\Pi_{\geq \lambda_{K}}G_{\mu}(\tau)
\end{align}	 
Since $\mu < \lambda_K$, the operator $(L+\mu)^{-j-1}$ is well defined by Proposition \ref{Prop;L-invertible}. 
Next, define
\begin{align}\label{formula;hat-v;13}
	\hat{v} =& e^{L\tau}\Pi_{\geq  \lambda_K}v(0) + \int_0^{\tau} e^{L(\tau-s)}\Pi_{\geq \lambda_K}\hat{G}(s)ds - \int_{\tau}^{\infty} e^{L(\tau-s)}\Pi_{< \lambda_K}\hat{G}(s)ds \nonumber \\  &+ \sum_{\mu \in \mathscr{S}\cap (0, \frac{m}{2n}]}e^{L\tau} \sum_{j=0}^{d_{\mu}} (L+\mu)^{-j-1} \frac{\partial^j}{\partial s^{j}}\Big|_{s=0}\Pi_{\geq \lambda_{K}}G_{\mu}
\end{align}
Since $\Pi_{\mu} \Pi_{\geq \lambda_{K}}\equiv 0$, by Proposition \ref{Prop;identity;2},  the last terms of \eqref{formula;v-m;12} and \eqref{formula;hat-v;13}, respectively, sums up to
\begin{align}\label{formula;hat-v;14}
	\sum_{\mu \in \mathscr{S}\cap (0, \frac{m}{2n}]} \int_0^{\tau}    e^{L(\tau-s)} e^{-\mu s}\Pi_{\geq \lambda_{K}}G_{\mu} (s) ds
\end{align}

Therefore, we obtain that
\begin{align}
	v_m + \hat{v} =  & e^{L\tau}v(0) + \int_0^{\tau} e^{L(\tau-s)}\hat{G}(s)ds  +\sum_{\mu \in \mathscr{S}\cap (0, \frac{m}{2n}]}\int_0^{\tau}e^{L(\tau-s)} \Pi_{<\lambda_K}\Big(e^{-\mu s}G_{\mu}(s)\Big)ds \nonumber \\
	&+\sum_{\mu \in \mathscr{S}\cap (0, \frac{m}{2n}]} \int_0^{\tau}    e^{L(\tau-s)} e^{-\mu s}\Pi_{\geq \lambda_{K}}G_{\mu} (s) ds \nonumber\\
	=& e^{L\tau}v(0) + \int_0^{\tau} e^{L(\tau-s)}\hat{G}(s)ds  +\sum_{\mu \in \mathscr{S}\cap (0, \frac{m}{2n}]}\int_0^{\tau}e^{L(\tau-s)}   e^{-\mu s}G_{\mu}(s) ds
\end{align}
 Comparing with \eqref{formula;v;11} we get:
\begin{align}\label{Eq;v-m-v;5}
	v = v_m + \hat{v} 
\end{align}

\noindent\textbf{Step2: }$v_m$ and $\hat{v}$ satisfy the right equation.

Let us compute
\begin{align}
	&(\partial_{\tau} - L) v_m \nonumber\\
	 =& (\partial_{\tau} - L) \Big[e^{L\tau}\Pi_{< \lambda_K}v(0) + \int_0^{\infty} e^{L(\tau-s)}\Pi_{< \lambda_K}\hat{G}(s)ds  \Big] \\
	&+(\partial_{\tau} - L) \Big[\sum_{\mu \in \mathscr{S}\cap (0, \frac{m}{2n}]}\int_0^{\tau}e^{L(\tau-s)} \Pi_{<\lambda_K}\Big(e^{-\mu s}G_{\mu}(s)\Big)ds \Big] \nonumber\\
	& - (\partial_{\tau} - L) \Big[ \sum_{\mu \in \mathscr{S}\cap (0, \frac{m}{2n}]}e^{-\mu\tau} \sum_{j=0}^{d_{\mu}} (L+\mu)^{-j-1}\partial_{\tau}^j\Pi_{\geq \lambda_{K}}G_{\mu}(\tau)\Big] \nonumber\\
	=& 0 +\sum_{\mu \in \mathscr{S}\cap (0, \frac{m}{2n}]}\Pi_{<\lambda_K}\Big(e^{-\mu \tau}G_{\mu}(\tau)\Big)  - \sum_{\mu \in \mathscr{S}\cap (0, \frac{m}{2n}]}e^{-\mu \tau} \sum_{j=0}^{d_{\mu}} (\partial_{\tau} - L-\mu)(L+\mu)^{-j-1}\partial_{\tau}^j\Pi_{\geq \lambda_{K}}G_{\mu}(\tau) \nonumber \\
	=& \sum_{\mu \in \mathscr{S}\cap (0, \frac{m}{2n}]} \Biggr\{\Pi_{<\lambda_K}\Big(e^{-\mu \tau}G_{\mu}(\tau)\Big) -  e^{-\mu \tau} \sum_{j=0}^{d_{\mu}} \Big[(L+\mu)^{-j-1}\partial_{\tau}^{j+1}\Pi_{\geq \lambda_{K}}G_{\mu}(\tau) -(L+\mu)^{-j}\partial_{\tau}^j\Pi_{\geq \lambda_{K}}G_{\mu}(\tau)  \Big]\Biggr\}
\end{align}
Since $\frac{\partial^{d_{\mu}+1}}{\partial s ^{d_{\mu}+1}}\Pi_{\geq \lambda_{K}}G_{\mu} (s) =0$, the last bracket simply becomes $e^{-\mu\tau}\Pi_{\geq \lambda_{K}}G_{\mu}(\tau)  $. This gives
\begin{align}\label{Eq;v-m;3}
	(\partial_{\tau} - L)v_m =&  \sum_{\mu \in \mathscr{S}\cap (0, \frac{m}{2n}]} e^{-\mu \tau}G_{\mu}(\tau) 
\end{align}
Subtracting from the equation of \eqref{Eq;v;3} and using the relation \eqref{estimate;hat-G;1}  and \eqref{Eq;v-m-v;5},  we get
\begin{align}\label{Eq;hat-v;2}
	(\partial_{\tau} - L) \hat{v} = \hat{G}
\end{align}

\noindent\textbf{Step3:} The estimate on error term $\hat{v}$. 

For each fixed $0\leq i \leq K-1$, we take the projection of the formula \eqref{formula;hat-v;13}, there is only one term $\displaystyle \Pi_{\lambda_i} \hat{v} = -\int_{\tau}^{\infty} e^{L(\tau-s)}\Pi_{ \lambda_i}\hat{G}(s)ds = -\int_{\tau}^{\infty} e^{-\lambda_i(\tau-s)}\Pi_{ \lambda_i}\hat{G}(s)ds $. Now we will estimate the decay. For each $\ell\geq 0$: 
\begin{align}
	||\Pi_{\lambda_i} \hat{v}||_{H^\ell} \leq  \int_{\tau}^{\infty} \big\Vert e^{-\lambda_i(\tau-s)}\hat{G}(s) \big\Vert_{H^\ell} ds  \leq \int_{\tau}^{\infty} e^{-\lambda_i(\tau-s)} O(e^{-\frac{m+1-\varepsilon}{2n}s}) ds = O(e^{-\frac{m+1-\varepsilon}{2n}\tau})
\end{align}
For the projection on higher eigenspaces, we apply the energy estimate to \eqref{Eq;hat-v;2}:
\begin{align}
	\frac{1}{2}\frac{d}{d\tau} ||\Pi_{\geq \lambda_K} \hat{v} ||_{H^\ell}^2 =& \langle L \Pi_{\geq \lambda_K} \hat{v}, \Pi_{\geq \lambda_K} \hat{v}\rangle_{H^\ell}  + \langle \Pi_{\geq \lambda_K} \hat{G}, \Pi_{\geq \lambda_K} \hat{v}\rangle_{H^\ell} \nonumber\\
	\leq& -\lambda_K||\Pi_{\geq \lambda_K} \hat{v}||_{H^\ell}^2  + || \hat{G}||_{H^\ell}||\Pi_{\geq \lambda_K} \hat{v}||_{H^\ell}
\end{align}
This in turn gives 
\begin{align}
	\frac{d}{d\tau} ||\Pi_{\geq \lambda_K} \hat{v} ||_{H^\ell} + \lambda_K  ||\Pi_{\geq \lambda_K} \hat{v} ||_{H^\ell} \leq ||\hat{G} ||_{H^\ell} = O(e^{-\frac{m+1-\varepsilon}{2n}\tau})
\end{align}
Using integrating factor and the fact that $ ||\Pi_{\geq \lambda_K} \hat{v} ||_{H^\ell}\rightarrow 0$ as $\tau \rightarrow\infty$, 
\begin{align}
	 ||\Pi_{\geq \lambda_K} \hat{v} ||_{H^\ell}  \leq &   e^{-\lambda_K\tau}\Big( ||\Pi_{\geq \lambda_K} \hat{v}(0)||_{H^\ell} + \int_{0}^{\tau} e^{\lambda_K s}\cdot O(e^{-\frac{m+1-\varepsilon}{2n}s}) ds \Big)\nonumber\\
	 \leq & C  e^{-\lambda_K\tau}  O\big(1+ e^{(\lambda_K -\frac{m+1-\varepsilon}{2n})\tau}\big)  
\end{align}
Since $\lambda_K \geq \frac{m+1}{2n}$, we get
\begin{align}
	||\Pi_{\geq \lambda_K} \hat{v} ||_{H^\ell}  = O(e^{-\frac{m+1-\varepsilon}{2n}\tau}) 
\end{align}
Adding them up we get for all $\ell \geq 0$:
\begin{align}
	||\hat{v}||_{H^\ell} \leq C \Big(||\Pi_{\geq \lambda_K} \hat{v} ||_{H^\ell} + \sum_{i=0}^{K-1} ||\Pi_{ \lambda_i} \hat{v} ||_{H^\ell} \Big) = O(e^{-\frac{m+1-\varepsilon}{2n}\tau}) 
\end{align}
By Sobolev embedding, this upgrades to
\begin{align}\label{estimate;hat-v;3}
	||\hat{v}||_{C^\ell} = O(e^{-\frac{m+1-\varepsilon}{2n}\tau}) 
\end{align} 
for all $\ell \geq 0$.
 
\noindent\textbf{Step4:} $v_m$ has the desired expansion.

Now we return to $v_m$, we will show that $v_m$ can be written as 
\begin{align} 
	 v_m = \sum_{\mu \in \mathscr{S}\cap (0, \frac{m}{2n}]}e^{-\mu\tau} F_{\mu} 
\end{align}
where $\{F_{\mu}\}_{\mu \in \mathscr{S}\cap (0, \frac{m}{2n}]} \subset \mathscr{F}$. 
We will decompose $v_m$ into several terms and do this one by one. 

\noindent  (1):
We set $F^{(1)}_{\mu} = \Pi_{\lambda_i} v(0) \in E_i \subset  \mathscr{F}$ if $\mu = \lambda_i$ for some $0\leq i \leq K-1$ and set  $F^{(1)}_{\mu} = 0$ otherwise. Then 
\begin{align}\label{formula;v-m;15}
	e^{L\tau}\Pi_{< \lambda_K}v(0) = \sum_{i=0}^{K-1} e^{-\lambda_{i}\tau}\Pi_{\lambda_i}v(0) = \sum_{i=0,1} e^{-\lambda_i \tau} F^{(1)}_{\lambda_i}(\tau)  + \sum_{\mu \in \mathscr{S}\cap (0, \frac{m}{2n}]} e^{-\mu\tau} F^{(1)}_{\mu}(\tau)
\end{align}
\noindent (2): We can compute
 \begin{align*}
 	\int_0^{\infty} e^{L(\tau-s)}\Pi_{< \lambda_K}\hat{G}(s)ds  = \sum_{i=0}^{K-1} \int_0^{\infty} e^{-\lambda_{i}(\tau-s) }\Pi_{\lambda_i}\hat{G}(s)ds =\sum_{i=0}^{K-1}  e^{-\lambda_{i}\tau} \int_0^{\infty} e^{\lambda_{i}s }\Pi_{\lambda_i}\hat{G}(s)ds
 \end{align*}
We have the estimate:
\begin{align*}
	\Big\Vert \int_0^{\infty}  e^{\lambda_{i} s }\Pi_{\lambda_i}\hat{G}(s)ds\Big\Vert_{H^\ell} \leq \int_0^{\infty}  e^{\lambda_{i} s }||\hat{G}(s)||_{H^\ell}ds = \int_0^{\infty}O(e^{(\lambda_i -\frac{m+1-\varepsilon}{2n})s})ds <\infty
\end{align*}
This means that $ \int_0^{\infty} e^{\lambda_{i}s }\Pi_{\lambda_i}\hat{G}(s)ds$ is well-defined and belongs to $\mathscr{F}$. 

We set $F^{(2)}_{\mu} = \int_0^{\infty} e^{\lambda_{i}s }\Pi_{\lambda_i}\hat{G}(s)ds \in E_i \subset \mathscr{F}$ if $\mu = \lambda_i$ for some $0\leq i \leq K-1$ and set  $F^{(2)}_{\mu} = 0$ otherwise. Then 
\begin{align}\label{formula;v-m;16}
	\int_0^{\infty} e^{L(\tau-s)}\Pi_{< \lambda_K}\hat{G}(s)ds = \sum_{i=0,1} e^{-\lambda_i \tau} F^{(2)}_{\lambda_i}(\tau)  + \sum_{\mu \in \mathscr{S}\cap (0, \frac{m}{2n}]} e^{-\mu\tau} F_{\mu}^{(2)}
\end{align}

\noindent(3) By Proposition \ref{Prop;identity;2}, since $\lambda_{K-1} \leq \frac{m}{2n}, $ we have
\begin{align}\label{formula;v-m;17}
	\sum_{\mu \in \mathscr{S}\cap (0, \frac{m}{2n}]}\int_0^{\tau}e^{L(\tau-s)} \Pi_{<\lambda_K}\Big(e^{-\mu s}G_{\mu}(s)\Big)ds =& \sum_{i=0}^{K-1} \sum_{\mu \in \mathscr{S}\cap (0, \frac{m}{2n}]}\int_0^{\tau}e^{L(\tau-s)} \Pi_{\lambda_i}\Big(e^{-\mu s}G_{\mu}(s)\Big)ds \nonumber\\
	 =&\sum_{i=0,1} e^{-\lambda_i \tau} F^{(3)}_{\lambda_i}(\tau) + \sum_{\mu \in \mathscr{S}\cap (0, \frac{m}{2n}]}e^{-\mu\tau} F_{\mu}^{(3)}
\end{align}
for some $F_{\lambda_0}^{(3)}, F_{\lambda_1}^{(3)} \in \mathscr{F}$ and  $\{F_{\mu}^{(3)}\}_{\mu \in \mathscr{S}\cap (0, \frac{m}{2n}]} \subset \mathscr{F}$.
 
\noindent(4) We set for all $\mu\in \mathscr{S}\cap (0, \frac{m}{2n}]$:
\begin{align*}
	F_{\mu}^{(4)} = -\sum_{j=0}^{d_{\mu}} (L+\mu)^{-j-1}\partial_{\tau}^j\Pi_{\geq \lambda_{K}}G_{\mu}(\tau).
\end{align*}
They are elements of $\mathscr{F}$. Then 
\begin{align}\label{formula;v-m;18}
	-\sum_{\mu \in \mathscr{S}\cap (0, \frac{m}{2n}]}e^{-\mu\tau} \sum_{j=0}^{d_{\mu}} (L+\mu)^{-j-1}\partial_{\tau}^j\Pi_{\geq \lambda_{K}}G_{\mu}(\tau) = \sum_{\mu \in \mathscr{S}\cap (0, \frac{m}{2n}]}e^{-\mu\tau}  F_{\mu}^{(4)} \in \mathscr{E}.
\end{align}
Now we set for all $\mu\in \mathscr{S}\cap (0, \frac{m}{2n}]$:
\begin{align}
	F_{\mu}:= \sum_{\alpha=1}^{4} F_{\mu}^{(\alpha)}.
\end{align}
Then in view of formula \eqref{formula;v-m;15} - \eqref{formula;v-m;18} and the expression of $v_m$ in \eqref{formula;v-m;12}, we get
\begin{align}\label{formula;v-m;19}
	v_m =\sum_{\substack{i=0,1\\j=1,2,3}} e^{-\lambda_i \tau} F^{(j)}_{\lambda_i}(\tau)  + \sum_{\alpha=1}^{4} \sum_{\mu \in \mathscr{S}\cap (0, \frac{m}{2n}]}e^{-\mu\tau}  F_{\mu}^{(\alpha)} =  \sum_{\substack{i=0,1\\j=1,2,3}} e^{-\lambda_i \tau} F^{(j)}_{\lambda_i}(\tau)  + \sum_{\mu \in \mathscr{S}\cap (0, \frac{m}{2n}]}e^{-\mu\tau}  F_{\mu}.
\end{align}

By \eqref{Eq;v-m-v;5}, \eqref{estimate;hat-v;3} and the assumption that $v\rightarrow 0$ in $C^{\infty}$,  we get:
\begin{align}
	\lim_{\tau\rightarrow\infty}||v_m||_{C^0} \leq \lim_{\tau\rightarrow\infty}||v||_{C^0} + \lim_{\tau\rightarrow\infty}||\hat{v}||_{C^0}  =0
\end{align}
In view of \eqref{formula;v-m;19}  and the fact that $\lambda_0, \lambda_1 <0$, this forces
\begin{align}
	  \sum_{j=1}^3 F^{(j)}_{\lambda_i}(\tau) = 0,\quad i=0,1
\end{align}
This gives the refinement:
\begin{align}\label{formula;v-m;20}
	v_m = \sum_{\mu \in \mathscr{S}\cap (0, \frac{m}{2n}]}e^{-\mu\tau}  F_{\mu} 
\end{align}
By  \eqref{Eq;v-m-v;5}, \eqref{estimate;hat-v;3} and  \eqref{formula;v-m;20}, the first assertion follows. 

To prove the second assertion \eqref{Eq;F-G;2}, one can rewrite \eqref{Eq;v-m;3} as 
\begin{align}
	\sum_{\mu \in \mathscr{S}\cap (0, \frac{m}{2n}]} e^{-\mu\tau}(\partial_{\tau} - (L+\mu))F_{\mu}  =  \sum_{\mu \in \mathscr{S}\cap (0, \frac{m}{2n}]} e^{-\mu \tau}G_{\mu}(\tau) 
\end{align}
By Proposition \ref{Prop;E-F-derivative;2}, for all $\mu \in \mathscr{S}\cap(0,\frac{m}{2n}]$ we have $(\partial_{\tau} - (L+\mu))F_{\mu}\in \mathscr{F}$ and therefore equals $G_{\mu}$  by unique representation. This is exactly \eqref{Eq;F-G;2}. 
\end{proof}
\bigskip
\bigskip
Now we are in the right position to prove Theorem \ref{Thm;asymptotic;2}.
\begin{proof}[Proof of Theorem \ref{Thm;asymptotic;2}]

Let us fix an arbitrary $\varepsilon\in (0,\frac{1}{4n})$. The statement clearly holds with larger $\varepsilon$. 
 
Suppose that $v \not\equiv 0$, for otherwise the result is trivial. Then by the work of Strehlke \cite[Theorem 2.2]{str2} and \cite[Theorem 2.1]{str}, there exists an integer $k\geq 2$,  $H_k \in E_k$ such that for all $\ell\geq 0$ and $\sigma < \lambda_k$:
\begin{align}\label{asymptotic;v;2}
	||v(\tau) - e^{-\lambda_k \tau} H_k ||_{H^\ell} = O(e^{-\lambda_{k+1}\tau} + e^{-2\sigma \tau})
\end{align}
In particular, we have for all $\ell \geq 0$:
\begin{align}\label{asymptotic;v;3}
	||v(\tau)||_{H^\ell} = O(e^{-\frac{2}{2n}\tau}),  
\end{align}
This also implies that 
\begin{align}
	v\rightarrow 0 \text{ in } C^{\infty} \text{ as } \tau \rightarrow \infty
\end{align}
 
Since the smallest element in $\mathscr{S}$ is $\lambda_2= \frac{2}{2n}$, for $m=0,1$ the summation in the Theorem is empty. The result then holds by Sobolev embedding. Set
\begin{align}\label{Def;F;3}
	F_{\lambda_2} = F_{\frac{2}{2n}} :=\begin{cases}
		H_k &\text{ if }  k = 2\\
		0 &\text{ if } k\geq 3.
	\end{cases}
\end{align}
Let us take $\sigma = \lambda_k - \varepsilon$. Since $\lambda_{k+1} \geq\lambda_3 > \frac{3}{2n} $ and $2\sigma = 2\lambda_k - 2 \varepsilon\ge 2\lambda_2 - 2 \varepsilon >   \frac{3}{2n}$. Then \eqref{asymptotic;v;2} and Sobolev embedding imply that the result is true for $m=2$. 

We then proceed by induction. Let $m\geq 3$ and suppose that the result holds for $m-1$ with coefficients $F_{\mu}\in \mathscr{F}$ for $\mu \in \mathscr{S} \cap (0,\frac{m-1}{2n}]$.  
Define
\begin{align}\label{asymptotic;v-m;4}
	v_{m-1} =  \sum_{\substack{\mu  \in \mathscr{S} \\ \mu \leq \frac{m-1}{2n}}}e^{-\mu\tau}F_{\mu}(\tau) \quad  \text{ and } \quad R_{m-1} = v - v_{m-1}
\end{align}
Then the closed expression of $v_{m-1}$ and our induction assumption implies that for all $\ell\geq 0$:
\begin{align}
	||v_{m-1}||_{C^\ell} = O(e^{-\frac{2}{2n}\tau}) \quad  \text{ and } \quad  ||R_{m-1}||_{C^\ell} = O(e^{-\frac{m-\varepsilon}{2n}\tau})
\end{align}  
In particular, for any fixed $\ell\geq 0$, we have $||v||_{C^{\ell+2}} + ||v_{m-1}||_{C^{\ell+2}} + ||R_{m-1}||_{C^{\ell+2}}  \leq \varepsilon_1$ for all sufficiently large $\tau$. Then we apply Lemma \ref{Lem;N-diff;1} with $a= 0$ to get:
\begin{align}\label{estimate;N;11}
	||\mathcal{N}(v) - \mathcal{N}(v_{m-1})||_{C^\ell} \leq & C_\ell\cdot  \Big(||v(\tau)||_{C^{\ell+2}} + ||v_{m-1}(\tau)||_{C^{\ell+2}}\Big) ||v(\tau) - v_{m-1}(\tau)||_{C^{\ell+2}}  \nonumber \\
	=& O(e^{-\frac{2}{2n}\tau}) \cdot O(e^{-\frac{m-\varepsilon}{2n}\tau}) =  O(e^{-\frac{m+2-\varepsilon}{2n}\tau}) 
\end{align}
Next, we apply Lemma \ref{Lem;N-tail;1} with $a= 0$ to get, for all sufficiently large $\tau$:
\begin{align}\label{estimate;N;12}
    ||\mathcal{N}(v_{m-1}) - \mathcal{N}_{\leq m}(v_{m-1})||_{C^\ell} =& ||\mathcal{N}_{> m }(v_{m-1})||_{C^\ell} \leq  C_{m,0,\ell}  ||v_{m-1}||_{C^{\ell+2}}^{m+1} = O(e^{-\frac{2m+2}{2n}\tau})
\end{align}

Since $v_{m-1}\in \mathscr{E}$, Lemma \ref{Lem;N-m;1} implies that $\mathcal{N}_{\leq m}(v_{m-1})\in \mathscr{E}$. Therefore, we can apply Definition  \ref{Def;function-class;1}  to define for all $\mu \in \mathscr{S}$:
\begin{align}\label{Def;G;11}
	G_{\mu} := [\mathcal{N}_{\leq m}(v_{m-1})]_{\mu}
\end{align}
Note that the definition of $\mathscr{E}$ (see Definition \ref{Def;function-class;1}) forces all but \textbf{finite} $G_{\mu}$ to vanish.  Since $\mathscr{S}\subset \frac{1}{2n}\mathbb{Z}$  we conclude that:
\begin{align}\label{estimate;N;13}
	\Big\Vert\mathcal{N}_{\leq m}(v_{m-1})-\sum_{\substack{\mu \in \mathscr{S} \cap (0,\frac{m}{2n}]}} e^{-\mu\tau} G_{\mu} \Big\Vert_{C^\ell}  = \Big\Vert \sum_{\substack{\mu  \in \mathscr{S}, \mu \geq \frac{m+1}{2n}}} e^{-\mu\tau} G_{\mu} \Big\Vert_{C^\ell}  = O(e^{-\frac{m+1-\varepsilon}{2n}\tau})
\end{align}
 
Adding \eqref{estimate;N;11}, \eqref{estimate;N;12} and \eqref{estimate;N;13} we get
\begin{align}
	\Big\Vert\mathcal{N}(v) - \sum_{\mu \in \mathscr{S} \cap (0,\frac{m}{2n}]} e^{-\mu\tau} G_{\mu} \Big\Vert_{C^\ell} = O(e^{-\frac{m+1-\varepsilon}{2n}\tau})
\end{align}
Now fix any $\varepsilon \in (0,1)$ and apply Lemma \ref{Lem;linear;1} with $\mathcal{N}(v)$ in place of $G$:
\begin{align}
	\big\Vert v - \sum_{\mu \in \mathscr{S} \cap (0,\frac{m}{2n}]} e^{-\mu\tau} \tilde{F}_{\mu} \big\Vert_{C^\ell} = O(e^{-\frac{m+1-\varepsilon}{2n}\tau})
\end{align}
where 
\begin{align}\label{Eq;F;5}
	\partial_{\tau} \tilde{F}_{\mu} - (L+\mu) \tilde{F}_{\mu} =   G_{\mu}
\end{align}
and $\tilde{F}_{\mu} \in \mathscr{F}$ for all $\mu \leq \frac{m}{2n}$ with $\mu\in \mathscr{S}$. It is clear that $\tilde{F}_{\mu}$ is not changed if we reduce $\varepsilon >0$.  

The result then holds with $m$. Comparing with induction hypothesis and the asymptotics,  we observe that $\tilde{F}_{\mu} = F_{\mu}$ for all $\mu \in \mathscr{S}\cap (0, \frac{m-1}{2n}]$. If $\frac{m}{2n} \in \mathscr{S}$, define $F_{\frac{m}{2n}} := \tilde{F}_{\frac{m}{2n}}$, otherwise define $F_{\frac{m}{2n}} := 0$. Then the induction step only gives one new element $F_{\frac{m}{2n}}$.  Thus our induction gives a single sequence $\{ F_{\mu} \}_{\mu \in \mathscr{S}} \subset \mathscr{F}$. 
The uniqueness of $\{ F_{\mu} \}_{\mu \in \mathscr{S}} \subset \mathscr{F}$ again follows by comparing asymptotics.

Finally,  we plug \eqref{asymptotic;v-m;4} into \eqref{Def;G;11}, then using uniqueness of $F_{\mu}$, the equation \eqref{Eq;F;5} is exactly the equation \eqref{Eq;F;3} for all $\mu\in \mathscr{S}$. For $m\geq 2$ with $\frac{m}{2n}\not\in \mathscr{S}$, we have defined $F_{\frac{m}{2n}} = 0$ and the RHS of \eqref{Eq;F;3} is 0 since $\mathcal{N}_{\leq m}(v_{m-1})\in \mathscr{E}$. Therefore \eqref{Eq;F;3} still holds. This completes the proof. 
\end{proof}

\begin{remark}\label{rem;induction}
    The uniqueness of $\{ F_{\mu} \}_{\mu \in \mathscr{S}}$ and our induction process implies that \eqref{Eq;F;5} holds for $F_{\mu}$. Namely for all $\mu\in \mathscr{S}$:
\begin{align}\label{Eq;F;6}
	\partial_{\tau} {F}_{\mu} - (L+\mu) {F}_{\mu} =   [\mathcal{N}_{\leq m}(v_{m-1})]_{\mu}
\end{align}
for all $m\in \mathbb{Z}$ with $\frac{m}{2n}\ge \mu$.
\end{remark}

We finish this section by showing that Theorem \ref{Thm;asymptotic;2} continues to hold if one incorporates time derivatives. Theorem \ref{t:intro graph} is a direct consequence of the following Corollary:
\begin{corollary}\label{Cor;asymptotic;3}
	Suppose that $v$ solves \eqref{Eq;RMCF;2} and $v \rightarrow 0$ in $C^{\infty}$ as $\tau \rightarrow \infty$. Then there exists a unique family $\{F_{\mu}\}_{\mu \in \mathscr{S}}\subset \mathscr{F}$ such that for all integers $m, \ell, a\geq 0$ and $\varepsilon \in (0,1)$,   
\begin{align}
	\Big\Vert \partial_{\tau}^a \Big(v - \sum_{\substack{\mu  \in \mathscr{S} \\ \mu \leq \frac{m}{2n}}}e^{-\mu\tau}F_{\mu}(\tau) \Big)\Big\Vert_{C^\ell(\Sigma)} = O(e^{-\frac{m+1-\varepsilon}{2n}\tau})
\end{align}
\end{corollary}
 
 \begin{proof} 
 We shall take $\{ F_{\mu} \}_{\mu \in \mathscr{S}} \subset \mathscr{F}$ from Theorem \ref{Thm;asymptotic;2} and set for all $j\geq 0$:
 \begin{align}\label{asymptotic;v-m;6}
	  v_{j} =  \sum_{\substack{\mu  \in \mathscr{S} \cap (0,\frac{j}{2n}]}} e^{-\mu\tau}F_{\mu}(\tau) \quad  \text{ and } \quad R_{j} = v - v_{j}
\end{align}
 By direct computation, $||\partial_{\tau}^a(e^{-\mu\tau}F_{\mu}) ||_{C^\ell} = O(e^{-(\mu - \frac{\varepsilon}{2n})\tau})$. Therefore, if the conclusion holds for $m$, then it holds for $0,1,\cdots, m-1$. As such, we only prove the statement for $m\geq 2$, which is equivalent to proving 
 \begin{align}
 	\Big\Vert \partial_{\tau}^aR_m\Big\Vert_{C^{\ell}} = O(e^{-\frac{m+1-\varepsilon}{2n}\tau})
 \end{align}
By Remark \ref{rem;induction}  we have 
\begin{align}
	 \partial_{\tau} v_m - Lv_m =& \sum_{\mu  \in \mathscr{S} \cap (0,\frac{m}{2n}]} e^{-\mu\tau}\big(\partial_{\tau} F_{\mu} - (L+\mu)F_{\mu}) \\
	 =& \sum_{\mu  \in \mathscr{S} \cap (0,\frac{m}{2n}]} e^{-\mu\tau} [\mathcal{N}_{\leq m}(v_{m-1})]_{\mu}  
\end{align}

Subtracting from the equation of $v$ we get the equation for $R_m$:
\begin{align}\label{Eq;R;2}
	\partial_{\tau} R_m - L R_m = \mathcal{N}(v) - \sum_{\mu  \in \mathscr{S} \cap (0,\frac{m}{2n}]} e^{-\mu\tau} [\mathcal{N}_{\leq m}(v_{m-1})]_{\mu}   = I + II + III
\end{align}
where 
\begin{align}
	& I = \mathcal{N}(v) - \mathcal{N}(v_{m-1}), \quad II = \mathcal{N}(v_{m-1}) - \mathcal{N}_{\leq m}(v_{m-1})  = \mathcal{N}_{> m}(v_{m-1}) \\
	& III = \mathcal{N}_{\leq m}(v_{m-1}) - \sum_{\mu  \in \mathscr{S} \cap (0,\frac{m}{2n}]} e^{-\mu\tau} [\mathcal{N}_{\leq m}(v_{m-1})]_{\mu}
\end{align}
We will proceed by induction on $a$. First, Theorem \ref{Thm;asymptotic;2} already establishes the case without time derivatives for any $\ell\geq 0$. This establishes the induction base.
For $a\geq 1$, assume that the result holds for $0,\cdots, a-1$ and any $\ell\geq 0$, we wish to establish the result for $a$ and any $\ell\geq 0$. 

Let us fix arbitrary $\ell \geq 0$ from now.  We establish rough asymptotics first. Direct computation shows that
\begin{align}\label{estimate;v-m;3}
	\Big\Vert\partial_{\tau}^j v_m\Big\Vert_{C^{\ell+2}} \leq \sum_{\substack{\mu  \in \mathscr{S} \cap (0,\frac{m}{2n}]}}  e^{-\mu\tau} ||(\partial_{\tau} - \mu)^jF_{\mu}(\tau)||_{C^{\ell+2}} = O(e^{-\frac{2}{2n}\tau})
\end{align}
 for all $j\geq 0$, where we used that $F_{\frac{2}{2n}}$ is time-independent by  \eqref{Def;F;3} . The same estimate holds for $v_{m-1}$. Together with induction hypothesis we get for $j=0,\cdots, a-1$ that:
\begin{align}\label{estimate;v;4}
	\Big\Vert\partial_{\tau}^j v\Big\Vert_{C^{\ell+2}} \leq \Big\Vert\partial_{\tau}^j R_m\Big\Vert_{C^{\ell+2}}  + \Big\Vert\partial_{\tau}^j v_m\Big\Vert_{C^{\ell+2}}  =O(e^{-\frac{m+1-\varepsilon}{2n}\tau}) + O(e^{-\frac{2}{2n}\tau})  = O(e^{-\frac{2}{2n}\tau}) 
\end{align}
Therefore, we can find $\mathcal{T}_* \geq 0$ such that $\sum_{j=0}^{a-1} (||\partial_{\tau}^j v_{m-1}||_{C^{\ell+2}} + ||\partial_{\tau}^j v||_{C^{\ell+2}}) < \varepsilon_1$ for all $\tau \geq \mathcal{T}_*$. 

Next,  we take $\partial_{\tau}^{a-1}$ on \eqref{Eq;R;2}:
\begin{align}\label{Eq;R;3}
	\partial_{\tau}^a R_m = \partial_{\tau}^{a-1} L R_m + \partial_{\tau}^{a-1}(I+II + III) 
\end{align}
We will estimate the $C^\ell$ norm for each part in the RHS.

\noindent \textbf{(1)}
Since $L$ commutes with $\partial_{\tau}$, by induction hypothesis we get  
\begin{align}\label{estimate;11}
	 \Big\Vert\partial_{\tau}^{a-1} L R_m\Big\Vert_{C^\ell} = \Big\Vert L\Big(\partial_{\tau}^{a-1}   R_m\Big)\Big\Vert_{C^\ell} \leq  C_\ell \Big\Vert \partial_{\tau}^{a-1}   R_m\Big\Vert_{C^{\ell+2}}  = O(e^{-\frac{m+1-\varepsilon}{2n}\tau})
\end{align}

\noindent  \textbf{(2)}
Since $\sum_{j=0}^{a-1} (||\partial_{\tau}^j v_{m-1}||_{C^{\ell+2}} + ||\partial_{\tau}^j v||_{C^{\ell+2}}) < \varepsilon_1$ for all $\tau \geq \mathcal{T}_*$, we can apply Lemma \ref{Lem;N-diff;1} to get:
\begin{align}\label{estimate;remainder;11}
	\Big\Vert \partial_{\tau}^{a-1} I \Big\Vert_{C^\ell} =& \Big\Vert \partial_{\tau}^{a-1}\big(\mathcal{N}(v) -\mathcal{N}(v_{m-1})\big) \Big\Vert_{C^\ell}  \nonumber\\
	 \leq& C_{\ell,a,m} \sum_{j+k\leq a-1}\Big(\big\Vert \partial_{\tau}^jv\big\Vert_{C^{\ell+2}} + \big\Vert \partial_{\tau}^jv_{m-1}\big\Vert_{C^{\ell+2}}\Big)\big\Vert\partial_{\tau}^k(v - v_{m-1} ) \big\Vert_{C^{\ell+2}} 
\end{align}
Note that $v- v_{m-1} = R_m + e^{-\frac{m}{2n}\tau}F_{\frac{m}{2n}}$, using \eqref{estimate;v-m;3} and \eqref{estimate;v;4} and induction hypothesis we get:
\begin{align}\label{estimate;12}
	\Big\Vert \partial_{\tau}^{a-1} I \Big\Vert_{C^\ell}  \leq& C_{\ell,a,m} \sum_{j =0}^{a-1}\Big\Vert\partial_{\tau}^j\Big( R_m + e^{-\frac{m}{2n}\tau}F_{\frac{m}{2n}}\Big) \Big\Vert_{C^{\ell+2}} \cdot Ce^{-\frac{2}{2n}\tau} \nonumber \\
      \leq&  \Big( O (e^{-\frac{m+1-\varepsilon}{2n}\tau}) +  O (e^{-\frac{m-\varepsilon}{2n}\tau})  \Big)\cdot O\Big(e^{-\frac{2}{2n}\tau} \Big) \nonumber\\
      =& O(e^{-\frac{m+2-\varepsilon}{2n}\tau}) 
\end{align} 

\noindent  \textbf{(3)} Since $\sum_{j=0}^{a-1}||\partial_{\tau}^j v_{m-1}||_{C^{\ell+2}}   < \varepsilon_1$ for all $\tau \geq \mathcal{T}_*$, we can apply Lemma \ref{Lem;N-tail;1} and \eqref{estimate;v-m;3} to get:
\begin{align}\label{estimate;13}
	\Big\Vert \partial_{\tau}^{a-1} II \Big\Vert_{C^\ell}  \leq& C_{m, a,\ell}  \sum_{a_1+\cdots+a_{m+1} \leq a-1} \prod_{i=1}^{m+1} \big\Vert \partial_{\tau}^{a_i}v_{m-1}\big\Vert_{C^{\ell+2}} = O(e^{-\frac{2m+2}{2n}\tau})
\end{align} 

\noindent  \textbf{(4)} Since $v_{m-1} \in \mathscr{E}$, Lemma \ref{Lem;N-m;1} implies that $\mathcal{N}_{\leq m}(v_{m-1})\in \mathscr{E}$. Therefore, 
\begin{align}
	III =  \sum_{\mu  \in \mathscr{S} , \mu\geq \frac{m+1}{2n}} e^{-\mu\tau} [\mathcal{N}_{\leq m}(v_{m-1})]_{\mu} 
\end{align}
where the summation is finite and $[\mathcal{N}_{\leq m}(v_{m-1})]_{\mu} \in \mathscr{F}$ for all $\mu \in \mathscr{S}$. Consequently, 
\begin{align}\label{estimate;14}
	\Big\Vert \partial_{\tau}^{a-1} III \Big\Vert_{C^\ell} \leq C_{m,a,\ell,\varepsilon} e^{-\frac{m+1-\varepsilon}{2n}\tau}
\end{align}
Adding \eqref{estimate;11}, \eqref{estimate;12}, \eqref{estimate;13}, \eqref{estimate;14} and applying \eqref{Eq;R;3} we get
\begin{align}
	\Big\Vert \partial_{\tau}^a R_m \Big\Vert_{C^\ell} = O(e^{-\frac{m+1-\varepsilon}{2n}\tau})
\end{align} 
Since $\ell$ is arbitrary, we proved the statement for $a$ and any $\ell\geq 0$. This finishes the induction step. Therefore, the Theorem is proved.

 \end{proof}

\section{Refined asymptotic expansion in the plane}

In this section, we focus on the planar case $n=1$ and study the refined asymptotic expansion for the graph functions on $S^1(\sqrt{2})$ for rescaled curve shortening flow.

When $n=1$, the linearized operator $L$, the eigenvalues and their eigenspaces are
\begin{equation}
    L= \frac{1}{2} \partial_{\theta\theta} +1, \quad \lambda_k = \frac{k^2-2}{2}, \quad E_k = \{ Re(ce^{ik\theta}):c \in \mathbb{C}\}.
\end{equation}

For $m\ge 2 $, we define the two sets of all possible frequencies for the term decaying as $e^{-m\tau/2}$
\begin{align}\label{e:plane cI}
    \mathcal{I}(m) = \big\{ \sum_{j=1}^d    \epsilon_j k_j : d\geq 1, k_j \ge 2, \epsilon_j \in \{-1,1\} \text{ and } \sum_{j=1}^{d} (k_j^2-2) = m. \big\}\\
    \mathcal{I}'(m) = \big\{ \sum_{j=1}^d    \epsilon_j k_j : d\geq 2, k_j \ge 2, \epsilon_j \in \{-1,1\} \text{ and } \sum_{j=1}^{d} (k_j^2-2) = m. \big\}
\end{align}
Clearly, $\mathcal{I}'(m)\subset \mathcal{I}(m)$.

The following Lemma provides some fundamental properties for $\cI(m)$: 
\begin{lemma}\label{l:l=m mod 2}
    If $l\in \cI(m)$, then $|l|\le m$ and $l\equiv m \text{ mod } 2$.
\end{lemma}

\begin{proof}
    Write $l = \sum_j \epsilon_j k_j$. Then $|l| \le \sum_j k_j \le \sum_j (k_j^2 -2) = m$ since $k_j \ge 2$. The second property follows from $m \equiv \sum_j k_j^2 \equiv \sum_j k_j \equiv \sum_j \epsilon_j k_j$ mod $2$. 
\end{proof}

\begin{lemma}\label{Lem;complex-product}
	If $F  = \sum_{l\in \mathcal{I}(m)} Re(c_{l}e^{i l \theta})$ and $G  = \sum_{l\in \mathcal{I}(m')} Re( c'_{l}e^{i l \theta})$, then 
	\begin{align}
		F G = \sum_{l \in \mathcal{I}'(m+ m')}Re(c''_l e^{il \theta})
	\end{align}
	for some $c''_l \in \mathbb{C}$,  $l\in \mathcal{I}'(m+ m')$.
\end{lemma}
\begin{proof}
	This follows from direct computations. 
	We can write 
	\begin{align}
		F = \sum_{l\in \mathcal{I}(m)} \frac{1}{2}(c_{l}e^{i l \theta} + \bar{c}_l e^{-il\theta}), \quad G = \sum_{l\in \mathcal{I}(m')} \frac{1}{2}(c'_{l}e^{i l \theta} + \bar{c'}_l e^{-il\theta})
	\end{align}
	Then the product gives
	\begin{align}
		FG =& \frac{1}{4}\sum_{l\in \mathcal{I}(m), l'\in \mathcal{I}(m')} c_l c'_{l'} e^{i(l+l')\theta} + \bar{c}_l \bar{c'}_{l'} e^{-i(l+l')\theta} + c_l \bar{c'}_{l'} e^{i(l-l')\theta} + \bar{c}_l c'_{l'} e^{i(-l+l')\theta}\nonumber\\
		=& \frac{1}{2}\sum_{l\in \mathcal{I}(m), l'\in \mathcal{I}(m')} Re(c_l c'_{l'} e^{i(l+l')\theta}) + Re(c_l \bar{c'}_{l'} e^{i(l-l')\theta}) 
	\end{align}
	For all $l\in \mathcal{I}(m)$ and $l'\in \mathcal{I}(m')$, it is easy to check that $l \pm l' \in \mathcal{I}'(m+m')$. The assertion then follows.  
\end{proof}

Next we prove a structure theorem for the coefficients of the asymptotic expansion for $v$ for $n=1$.

\begin{theorem}\label{t:planar graph expand}
    Suppose that $n=1$ and $v$ solves \eqref{Eq;RMCF;2} with $v \rightarrow 0$ in $C^{\infty}$ as $\tau \rightarrow \infty$. Take $\{F_{\mu}\}_{\mu\in \mathscr{S}}\subset \mathscr{F}$ from Theorem \ref{Thm;asymptotic;2}, then the following holds:
    \begin{enumerate}
        \item For $\mu<23/2$ or $\mu=12$, we have
        \begin{equation*}
            F_{\mu}=f_{\mu,0} \in \text{span}_{\RR} \{\cos(l\theta), \sin(l \theta) : l \in \cI(2\mu)\}.
        \end{equation*}
        
        \item We can decompose $F_{23/2} = f_{23/2,0} + \tau f_{23/2,1}$ satisfying
        \begin{equation*}
            f_{23/2,0}  \in \text{span}_{\RR} \{\cos(l\theta), \sin(l \theta) : l \in \cI(23)\} \text{ and } f_{23/2,1}  \in E_5.
        \end{equation*}
    \end{enumerate}
\end{theorem}
\begin{proof}
We first aim to prove  
\begin{align}\label{formula;F;3}
	F_{m/2}(\theta) = \sum_{l\in \mathcal{I}(m)} Re(c_{m,l}e^{i l \theta}), \quad c_{m,l} \in \mathbb{C}
\end{align}
for $2\leq m \leq 22$ and $m =24$ by induction. 

By Definition \ref{Def;F;3}, $F_{1}$ is an element of $E_2$, which must be in the form of $Re(c_{2,2}e^{2i\theta})$ for some $c_{2,2}\in \mathbb{C}$.  This establishes the induction base.

For $3\leq m\leq 22$ and $m=24$, suppose that \eqref{formula;F;3} is established for $2,\cdots, m-2$.  Let  
\begin{align}
	 v_{m-1} = \sum_{\substack{\mu  \in \mathscr{S} \cap (0,\frac{m-1}{2}]}} e^{-\mu\tau}F_{\mu}(\tau) \quad  \text{ and } \quad G_{m/2} = [\mathcal{N}_{\leq m}(v_{m-1})]_{\frac{m}{2}} 
\end{align}

Direct computation gives the rescaled mean curvature flow equation over $S^1(\sqrt{2})$:
\begin{align*}
    \partial_{\tau}v = \frac{v_{\theta\theta}}{(\sqrt{2} + v)^2+v_\theta^2}
 -\frac{(\sqrt{2} + v)^2+2v_\theta^2}{(\sqrt{2} + v)\big[(\sqrt{2} + v)^2+v_\theta^2\big]}+\frac{(\sqrt{2} + v)}{2}.
\end{align*}
The linearization of the right-hand side at $v=0$ is $L$. The nonlinear error is
\begin{align*}
    \mathcal{N}(v,v_{\theta}, v_{\theta\theta})  = - \frac{2\sqrt{2}v+v^2 + v_{\theta}^2}{2\big[(\sqrt{2}+v)^2 + v_{\theta}^2\big]}v_{\theta\theta} - \frac{v^2}{2(\sqrt{2}+v)} - \frac{v_{\theta}^2}{(\sqrt{2}+v)\big[(\sqrt{2}+v)^2 + v_{\theta}^2\big]}
\end{align*}
with $\mathcal N(0,0,0) = D\mathcal N(0,0,0)= 0$. Since the coefficients in $\mathcal{N}$ have no $\theta$ dependence, all Taylor coefficients $C_{a,b,c}$ below are constants:

\begin{align}
	G_{m/2} =& \Big[\sum_{2\leq a+b+c \leq m}C_{a,b,c} v_{m-1}^a (\partial_{\theta} v_{m-1})^b (\partial_{\theta\theta} v_{m-1})^c \Big]_{m/2}
\end{align}
Applying the expansion of $v_{m-1}$ we get:
\begin{align}\label{formula;G;4}
	G_{m/2} =&  \sum_{q=2}^{m}  \sum_{\substack{b_j \in \{0,1,2\}  \\ \mu_1 + \cdots + \mu_q = m/2 \\ \mu_j \in \mathscr{S}}} C_{q,\mu_1,b_1,\cdots, \mu_q, b_q} \prod_{j=1}^q  \partial_{\theta}^{b_j}F_{\mu_j}
\end{align}
Note that the summation is empty when $q > m/2 $. But we still keep the summation upper bound as $m$ for notation convenience.

Since the smallest element in $\mathscr{S}$ is $1$, the summation is empty if $m=3$. Moreover, the summation condition $\mu_1 + \cdots + \mu_q = m/2$ with $\mu_j \in \mathscr{S}$ and $q\geq 2$ implies that each $\mu_j\leq \frac{m-2}{2}$. Then for each summand $\prod_{i=1}^q  \partial_{\theta}^{b_i}F_{\mu_i}$,  
we can apply the induction hypothesis \eqref{formula;F;3} and iterate Lemma \ref{Lem;complex-product} to get 
\begin{align*}
    \prod_{j=1}^q  \partial_{\theta}^{b_j}F_{\mu_j} = \prod_{j=1}^q \sum_{l \in \mathcal{I}(2\mu_j)} Re(c_{2\mu_j, l} \cdot (il)^{b_j}\cdot e^{i l\theta}).
\end{align*}If we iterate Lemma \ref{Lem;complex-product}, this product must be in the form of 
\begin{align}
	\displaystyle \prod_{j=1}^q  \partial_{\theta}^{b_j}F_{\mu_j} =  \sum_{l\in \mathcal{I}'(2\mu_1 +\cdots +2\mu_q)} Re(c'_l e^{il\theta}), \qquad c'_l \in \mathbb{C}
\end{align} 
Plugging  into \eqref{formula;G;4}, we conclude that there exists $G_{m,l}\in \mathbb{C}$, $l\in \mathcal{I}'(m)$ such that
\begin{align}\label{formula;G;5}
	G_{m/2} = \sum_{l\in \mathcal{I}'(m)} Re(G_{m,l}e^{il\theta})
\end{align}

Recall that 
\begin{align}\label{Eq;F;4}
	\partial_{\tau}F_{m/2} - \big(L+m/2\big)F_{m/2} = G_{m/2}
\end{align}

Note that the only $2\lambda_k$ in $[3, 22]\cup \{24\}$ are $7,14$. 
If $m \neq 7, 14$, then
\begin{align}\label{formula;F;4}
	F_{m/2} =\sum_{l \in \mathcal{I}'(m)} \frac{2}{l^2-m-2} Re(G_{m,l}e^{il\theta})
\end{align}
Indeed, one can directly check that RHS solves \eqref{Eq;F;4}. Then by Proposition \ref{Prop;unique-linear} this is the only possibility, thus must coincide with $F_{m/2}$. 

We shall examine the definition of $\mathcal{I}'(m)$ more closely in the remaining cases. 
Note that $\sum_{j=1}^{d}(k_j^2-2) = m$, $k_j \geq 2$ and $d\geq 2$ will force each $k_j^2-2 < m$. 

If $m=7$, the only possibly that $k_j\geq 2$ and $k_j^2-2 < 7$ is $k_j =2$. However, it is not possible to have $\sum_{j=1}^{d}(2^2-2) = 7$. Therefore, 
\begin{align}
	\mathcal{I}'(7) = \emptyset
\end{align}
By Proposition \ref{Prop;unique-linear} we have
 \begin{align}\label{formula;F;5}
	F_{7/2} = Re(c_{7,3}e^{3i\theta}) 
\end{align}
for some $c_{7,3}\in \mathbb{C}$.

If $m=14$, the only possibly that $k_j\geq 2$ and $k_j^2-2 < 14$ is $k_j =2,3$. Then we only have the following two possibilities:
\begin{itemize}
	\item $d = 7$ and $k_1 =\cdots = k_7 = 2$
	\item $d = 2$ and $k_1 = k_2 = 3$
\end{itemize} 
This gives 
\begin{align}
	\mathcal{I}'(14) = \{0,\pm 2,  \pm 6,   \pm 10,  \pm 14\}
\end{align}
In particular, $\pm 4\not\in \mathcal{I}'(14)$, therefore by Proposition \ref{Prop;unique-linear} we get:
\begin{align}\label{formula;F;6}
	F_7 =  \sum_{l \in \mathcal{I}'(14)} \frac{2}{l^2-16} Re(G_{14,l}e^{il\theta}) + Re(c_{14,4}e^{4i\theta})
\end{align}
for some $c_{14,4}\in \mathbb{C}$. Indeed, one can check that RHS solves \eqref{Eq;F;4}, while  Proposition \ref{Prop;unique-linear} implies that any other possibility differs by $Re(ce^{4i\theta})$ for some $c\in \mathbb{C}$, which is in the same form. 

Observe that \eqref{formula;F;4}, \eqref{formula;F;5}, \eqref{formula;F;6} are all time-independent and are in the form of \eqref{formula;F;3},  we conclude that the assertion \eqref{formula;F;3} holds for $m$, then the induction step is complete. 

So far, we have proved the result for $2\leq m\leq 22$ and $m=24$, it remains to study $m=23$. In this case, \eqref{formula;G;5}, \eqref{Eq;F;4} holds with $m=23$. Since $23/2 = \lambda_5$, by Proposition \ref{Prop;unique-linear} we conclude that
\begin{align}\label{formula;F;7}
	F_{23/2} =\sum_{\substack{l \in \mathcal{I}'(23) \\ l\neq \pm 5}} \frac{2}{l^2-25} Re(G_{23,l}e^{il\theta}) + \sum_{l=-5,5} \tau Re(G_{23,l}e^{il\theta}) + Re(c_{23,5}e^{5i\theta})
\end{align}
for some $c_{23,5}\in \mathbb{C}$. Indeed, one can check that RHS solves  \eqref{Eq;F;4}, while  Proposition \ref{Prop;unique-linear} implies that any other possibility must be in the same form.

The Theorem follows from \eqref{formula;F;3} and \eqref{formula;F;7} 
\end{proof}

\begin{corollary}\label{Cor:CSF graph expansion}

Let $f_{m/2,0}$ and $f_{23/2,1}$ be as in Theorem \ref{t:planar graph expand}. We have
\begin{equation}\label{e: CSF graph expansion}
 v(\tau)= \sum_{m=2}^{24}e^{-m\tau/2}f_{m/2,0}
       +\tau e^{-23\tau/2}f_{23/2,1}+R_{24}(\tau).
\end{equation}
Here $f_{3/2,0}=f_{5/2,0}=0$, and for every $\ell,a \geq 0$ and
$0< \varepsilon <1$,
\begin{equation} \label{e: CSF graph remainder}
 \|\partial_{\tau}^aR_{24}(\tau)\|_{C^{\ell}(\Sigma)}
       =O(e^{-(25-\varepsilon)\tau/2}).
\end{equation}

\end{corollary}

\section{Arrival time expansion and optimal regularity}

We now transfer the higher-order asymptotic expansion in Corollaries \ref{Cor;asymptotic;3} and \ref{Cor:CSF graph expansion} for the graph function $v$ of the rescaled mean curvature flow $\bar M_{\tau}$ over $\mathbb{S}^n (\sqrt{2n})$ back to the arrival time function $U(x)$ near the extinction point. Recall that the arrival time function satisfies the equation in the viscosity sense 
\begin{equation}\label{e:arrival time}
    \Delta U - \frac{D^2U(DU,DU)}{|DU|^2} = -1 \text{ in } \Omega \subset \RR^{n+1}.
\end{equation}

Given a smooth convex mean curvature flow $M_t$ in $\RR^{n+1}$. Assume $(0,T)$ be the space-time extinction point. We consider the following rescaling: 
\begin{equation*}
    \tau = - \log(T-t), \quad s= \sqrt{T -t}=e^{-\tau/2}.
\end{equation*}
Then we can write the coordinates of a point on the hypersurface in the following form
\begin{equation}\label{e:x and v, s}
    x = s \big(\sqrt{2n} + v(\sqrt{2n} \omega, -2\log s) \big) \omega, \quad \omega\in \mathbb{S}^n(1). 
\end{equation}
For the graph function $v(\sqrt{2n}\omega, \tau) = v(\sqrt{2n}\omega, -2\log s)$, by the chain rule we have
\begin{equation}\label{e:partial s to tau}
    s \partial_s \big( v(\sqrt{2n}\omega, -2\log s) \big) = -2 \partial_{\tau} v (\sqrt{2n}\omega, -2\log s). 
\end{equation}

\subsection{Optimal regularity in the plane}

In this subsection, we mainly focus on the planar case $n=1$. The main result is to obtain the higher-order Taylor expansion of the arrival time function $U(x)$ near $0$, thus proving Theorem \ref{t:intro planar expansion}:

\begin{theorem}\label{t:arrival higher-order expand}
Let $U$ be the arrival time of a smooth closed convex curve shortening flow shrinking to $(0,T)$. There exist a small constant $r_0$, real homogeneous polynomials $P_d$ of degree $d$ with $4\le d\le 26$, and a real homogeneous polynomial $P_{25}^{\log}$ of degree $25$ such that the following holds for $|x| \le r_0$
    \begin{equation}\label{e:U plane expand}
        T - U(x) = \frac{|x|^2}{2} + \sum_{d=4}^{26} P_d(x) + P_{25}^{\log}(x) \log|x| + \cR(x)
    \end{equation}
    where $\cR(x)$ is the remainder term satisfying that for any $\varepsilon>0$ and any multi-index $\beta$
    \begin{equation}
        |D^{\beta} \cR(x)| \le C_{\beta,\varepsilon} |x|^{27-\varepsilon-|\beta|}.
    \end{equation}
\end{theorem}

For any $n\ge 1$ and any smooth function $f(\omega,s)$ on $\mathbb{S}^n(1) \times (0,s_0)$, we write $f = O_*(s^q)$ if 
\begin{equation}
    \| (s\partial_s)^a f(\cdot, s)\|_{C^\ell(\mathbb{S}^n(1))} \le C_{a,\ell}s^q \text{ for any } a,\ell\ge 0.
\end{equation}

We prove an inverse formula Lemma that works for any dimension.

\begin{lemma}\label{l:inverse formula}
    Let $n\ge 1$ and $\omega \in \mathbb{S}^n(1)$. Suppose $r(\omega,s) = \sqrt{2n} s + O_*(s^{1+\alpha})$ for some $\alpha>0$. Then for sufficiently small $r>0$, there exists a unique inverse function $s = s(\omega,r)$ satisfying
    \begin{equation*}
        s(\omega,r) = \frac{r}{\sqrt{2n}} + O_*(r^{1+\alpha}).
    \end{equation*}
    Moreover, suppose there exists a smooth function $\tilde{s}>0$ for $r>0$ with $\tilde{s} = O_*(r)$ such that $r(\omega,\tilde{s}) - r = O_*(r^{\beta})$ for some $\beta>0$, then $s - \tilde{s} = O_*(r^{\beta})$.  
\end{lemma}

\begin{proof}
We write $R(\omega,s) = \sqrt{2n}s + E(\omega,s)$ with $E \in O_*(s^{1+\alpha})$. Since $\partial_s R = \sqrt{2n} +O_*(s^{\alpha})$, we have $ \sqrt{n}s \le R \le 2\sqrt{n}s$ and $\sqrt{n} \le \partial_s R \le 2\sqrt{n}$ for sufficiently small $s$. Hence there exists a unique inverse $s(\omega,r)$ such that $R(\omega,s(\omega,r))= r$. It is smooth for $r>0$. Next we prove the derivative estimates for $s$.

We use local coordinates $\omega = (\omega^1,..., \omega^n)$ on the sphere. By the chain rule we have 
\begin{align*}
    \partial_r s &= \frac{1}{\partial_s R} = \frac{1}{ \sqrt{2n} + \partial_s E},\\
    \partial_{\omega^i}s &= -\frac{\partial_{\omega^i} R}{\partial_s R} = - \frac{\partial_{\omega^i} E}{\sqrt{2n} + \partial_s E}.
\end{align*}

The definition of $O_*$ gives $|\partial_s^j \partial_\omega^\gamma  E|\leq C_{j,\gamma}s^{1+\alpha-j}$ for any $j\ge 0$ and multi-index $\gamma$. Repeatedly differentiating $R(\omega, s(\omega,r))=r$ now gives, by induction on $a+|\gamma|$,
\begin{equation*}
    |\partial_r^a\partial_\omega^\gamma s|\leq C_{a,\gamma}r^{1-a}.
\end{equation*}

Apply the same chain rule to $s(\omega,r) - R(\omega,s)/\sqrt{2n}= - E(\omega,s(\omega,r))/\sqrt{2n}$ yields
\begin{equation*}
    |\partial_r^a\partial_\omega^\gamma (s - r/\sqrt{2n})|\leq C_{a,\gamma}r^{1+\alpha-a}.
\end{equation*}

This proves the first part.  

To prove the second part, if $\beta \le 1$, then the conclusion follows from $s, \tilde{s}=O_*(r)$. We assume $\beta>1$. Note by the fundamental theorem of calculus that
\begin{equation*}
    R(\omega,\tilde{s}) - r = (\tilde{s} - s) ( \sqrt{2n} + \int_0^1 \partial_sE(\omega, t\tilde{s} + (1-t)s) dt ).
\end{equation*}

Since $E\in O_*(s^{1+\alpha})$, in particular we have $\int_0^1 \partial_sE(\omega, t\tilde{s} + (1-t)s) dt = O_*(r^{\alpha}).$ By assumption we have $R(\omega,\tilde{s}) - r = O_*(r^{\beta})$, Therefore
\begin{equation*}
    \tilde{s} - s = \frac{O_*(r^{\beta})}{\sqrt{2n}+O_*(r^{\alpha})} = O_*(r^{\beta}).
\end{equation*}
This completes the proof.
\end{proof}

Next we apply Lemma \ref{l:inverse formula} with $n=1$ to prove Theorem \ref{t:arrival higher-order expand}.

\begin{proof}[Proof of Theorem \ref{t:arrival higher-order expand}]
    For any point $x$ on the curve shortening flow, by \eqref{e:x and v, s} we have
    \begin{equation}
        r(\omega,s) = s\big( \sqrt{2} + v(\sqrt{2} \omega, -2\log s) \big).
    \end{equation}
In view of the expansion of $v$, \eqref{e: CSF graph expansion}, we define
\begin{equation}
    r_0(\omega,s) = \sqrt{2}s + \sum_{m=2}^{24} s^{m+1} f_{m/2,0}(\sqrt{2}\omega). 
\end{equation}
By \eqref{e: CSF graph remainder} we have
\begin{equation*}
    \|(s\partial_s)^a [s R_{24}(\sqrt{2}\omega,-2\log s)] \|_{C^\ell} \le C_{a} s \sum_{j=0}^a \|\partial_{\tau}^j R_{24}\|_{C^\ell} \le C_{a,\ell,\varepsilon}s^{26 -\varepsilon}.
\end{equation*}
Therefore 
\eqref{e: CSF graph expansion} implies that 
\begin{equation}\label{e: r-r_0}
    r = r_0 - 2 s^{24} \log s f_{23/2,1}(\sqrt{2}\omega)+ O_*(s^{26-\varepsilon}).
\end{equation}

By Lemma \ref{l:inverse formula}, there exists its inverse $s=\frac{r}{\sqrt{2}} + O_*(r^{3})$.  

\textbf{Step 1: approximate $r_0$: $r_0(\omega,s_0) -r = O_*(r^{26})$.}

We first construct a finite inverse for $r_0$. For a polynomial in $r$, we denote its coefficient of $r^j$ by $[\cdot]_{r^j}$. For any $3 \le j \le 25$, we define successively
\begin{equation}\label{e:inverse s and h}
\begin{split}
    s_{0,j-1}(\omega,r) &= \frac{r}{\sqrt{2}} + \sum_{i=3}^{j-1}h_{i}(\omega) r^i \\
    h_j(\omega) &= - \frac{1}{\sqrt{2}} [r_0(\omega, s_{0,j-1}(\omega,r)) - r]_{r^j}.
\end{split}
\end{equation}

Note that there is no need to consider the quadratic term $h_2(\omega)r^2$ since 
\begin{equation}\label{e:finite appro base step}
    r_0(\omega, r/\sqrt{2}) - r = \sum_{m=2}^{24} \frac{f_{m/2,0}(\sqrt{2}\omega)}{2^{\frac{m+1}{2}}}r^{m+1} = O_*(r^3). 
\end{equation}

\textbf{Claim: } $r_0(\omega,s_{0,j})-r = O_*(r^{j+1})$. 

We prove the claim by induction. The base case is proved in \eqref{e:finite appro base step}. We now assume the claim holds up to $j-1$. Then we have
\begin{align*}
    &\quad r_0(\omega,s_{0,j-1} +h_jr^j) - r_0(\omega,s_{0,j-1}) \\
    &= \sqrt{2} h_j r^j +\sum_{m=2}^{24} f_{m/2,0} \sum_{a=1}^{m+1}\binom{m+1}{a} s_{0,j-1}^{m+1-a} h_{j}^a r^{aj} \\
    &= \sqrt{2} h_j r^j + O_*(r^{j+1}).
\end{align*}
In the last equality we use that $m\ge 2$ and $s_0 = O_*(r)$. 

By the definition of $h_j$ and induction hypothesis, we have $h_j r^j = -\frac{1}{\sqrt{2}}( r_0(\omega,s_{0,j-1})-r) + O_*(r^{j+1})$. Combined with the equation above, it implies that
\begin{equation*}
    r_0(\omega,s_{0,j}) = r_0(\omega,s_{0,j-1} +h_jr^j) = r + O_*(r^{j+1}).
\end{equation*}

This finishes the induction and proves the claim. 

Let $j=25$. The claim then implies that $r_0(\omega,s_0) -r = O_*(r^{26})$ where we set $s_0 = s_{0,25}$. 

\textbf{Step 2: approximate $r$: $r(\omega,\tilde{s}) - r  = O_*(r^{26-\varepsilon})$.}

The polynomial construction above only gives the finite-order approximation for $r_0$. Now we correct the error term for the function $r$, i.e. the term involving $f_{23/2,1}$ in \eqref{e: r-r_0}. 

We set the function $\tilde{s}$ defined as 
\begin{equation*}
    \tilde{s} = s_0 + \sqrt{2} (\frac{r}{\sqrt{2}})^{24} \log \frac{r}{\sqrt{2}} f_{23/2,1}(\sqrt{2}\omega). 
\end{equation*}

The goal is to estimate the approximation $r(\omega,\tilde{s}) - r$. We break it into two parts:
\begin{align*}
    r(\omega,\tilde{s}) - r = \big( r(\omega,\tilde{s}) - r_0(\omega,\tilde{s})\big) + \big( r_0(\omega,\tilde{s}) - r \big).
\end{align*}

For the first term, by definitions of $r$ and $r_0$ we have
\begin{align*}
    r(\omega,\tilde{s}) - r_0(\omega,\tilde{s}) &= -2 \tilde{s}^{24} \log \tilde{s} f_{23/2,1}(\sqrt{2}\omega)+ O_*(r^{26-\varepsilon}).
\end{align*}

Since $s_0 = \frac{r}{\sqrt{2}} + O_*(r^3)$, we have 
\begin{align*}
    \tilde{s}^{24} &= (\frac{r}{\sqrt{2}})^{24} + O_*(r^{26}) \\
    \log \tilde{s} &= \log( \frac{r}{\sqrt{2}}) + \log ( 1+ O_*(r^2)) = \log( \frac{r}{\sqrt{2}}) + O_*(r^2)\\
    \tilde{s}^{24} \log \tilde{s} &=  (\frac{r}{\sqrt{2}})^{24} \log( \frac{r}{\sqrt{2}}) + O_*(r^{26-\varepsilon}).
\end{align*}

This implies that
\begin{equation*}
    r(\omega,\tilde{s}) - r_0(\omega,\tilde{s}) = -2(\frac{r}{\sqrt{2}})^{24} \log( \frac{r}{\sqrt{2}})f_{23/2,1}(\sqrt{2}\omega) + O_*(r^{26-\varepsilon})
\end{equation*}

For the second term, using the binomial formula and $\tilde{s} - s_0 = O_*(r^{24-\varepsilon})$ we have
\begin{align*}
    r_0(\omega,\tilde{s}) - r &= r_0(\omega,\tilde{s}) - r_0(\omega, s_0) + O_*(r^{26}) \\
    &= \sqrt{2} (\tilde{s} - s_0) + \sum_{m=2}^{24} \tilde{s}^{m+1} f_{m/2,0}(\sqrt{2}\omega) - \sum_{m=2}^{24} s_0^{m+1} f_{m/2,0}(\sqrt{2}\omega) + O_*(r^{26}) \\
    &= \sqrt{2} (\tilde{s} - s_0) + O_*(r^{26-\varepsilon}).
\end{align*}

Combining all estimates we have the term involving $f_{23/2,1}$ canceled out:
\begin{equation*}
    r(\omega,\tilde{s}) - r  = O_*(r^{26-\varepsilon}). 
\end{equation*}
This completes the proof of step 2.

Therefore, by Lemma \ref{l:inverse formula}, we have $s - \tilde{s} = O_*(r^{26-\varepsilon})$.
\begin{equation}\label{e:inverse U expand}
\begin{split}
    T - U(x) &= s^2 = (\tilde{s} + O_*(r^{26-\varepsilon}))^2  = \tilde{s}^2 + O_*(r^{27-\varepsilon}) \\
    &= s_0^2 + 2\sqrt{2} (\frac{r}{\sqrt{2}})^{25} \log \frac{r}{\sqrt{2}} f_{23/2,1}(\sqrt{2}\omega) + O_*(r^{27-\varepsilon}). 
\end{split}
\end{equation}

Write $[s_0^2 ]_{r^d} = q_d(\omega)$. Then we have
\begin{equation*}
    s_0^2 = \frac{|x|^2}{2} + \sum_{d=4}^{26} r^d q_d(\omega) + O_*(r^{27}). 
\end{equation*}

We define
\begin{align*}
    P_d &= r^d q_d(\omega) \text{ for } 4\le d \le 26 \text{ and } d \neq 25 \\
    P_{25} &= r^{25}q_{25}(\omega) - \sqrt{2} \log2 (\frac{r}{\sqrt{2}})^{25}  f_{23/2,1}(\sqrt{2}\omega) \\
    P_{25}^{\log} &=  2\sqrt{2} (\frac{r}{\sqrt{2}})^{25}   f_{23/2,1}(\sqrt{2}\omega)
\end{align*}

It remains to prove that $r^{d}q_d(\omega)$ and $r^{25}f_{23/2,1}(\sqrt{2}\omega)$ are all polynomials.  

First, by Theorem \ref{t:planar graph expand}, we have $f_{23/2,1} \in E_5$. Hence $f_{23/2,1}$ is a linear combination of $\sin(5\theta), \cos(5\theta)$ where $\omega=(\cos \theta, \sin \theta)$ and thus $r^{25} f_{23/2,1}$ is a polynomial in $\RR^2$. 

Next prove that $r^{d}q_d(\omega)$ are all polynomials. 

Set $h_1 = 1/\sqrt{2}$ and $h_2=0$. Then we can write $s_{0} = \sum_{j=1}^{25} h_j r^j$. For $j\ge3$, by \eqref{e:inverse s and h}, we can expand $s_0$ to obtain the recursive formula for $h_j$
\begin{equation}\label{e:h_j expansion}
\begin{split}
    h_j(\omega) &= - \frac{1}{\sqrt{2}} \sum_{m=2}^{j-1}f_{m/2,0}(\sqrt{2}\omega) \big[ (\sum_{i=1}^{j-1}h_i(\omega) r^i )^{m+1} \big]_{r^j} \\
    &= - \frac{1}{\sqrt{2}} \sum_{m=2}^{j-1}f_{m/2,0}(\sqrt{2}\omega) \sum_{\substack{\sum_{k=1}^{m+1} a_k = j \\ 1\le a_k \le j-1}} \prod_{k=1}^{m+1} h_{a_k}(\omega).
\end{split}
\end{equation}

Using induction  we obtain that
\begin{equation*}
    h_j(\omega) \in \text{span}\{ \prod_{k=1}^N f_{m_k/2,0}(\sqrt{2} \omega) : N \ge 1, m_k \le 24,   \sum_k m_k = j-1 .\}
\end{equation*}

Indeed, the formula $h_3 = - \frac{1}{4}f_{1,0}(\sqrt{2}\omega)$ finishes the base step. Then note that in \eqref{e:h_j expansion}, we have $ m+ \sum_{k=1}^{m+1} (a_k -1)  = (\sum_{k=1}^{m+1} a_k) -1 = j-1$. This completes the induction.

Consequently, we obtain the expansion for $s_0^2$:
\begin{equation*}
    q_d(\omega) = \sum_{j+k=d}  h_{j} h_k \in \text{span}\{ \prod_{k=1}^N f_{m_k/2,0}(\sqrt{2} \omega) : N \ge 1, m_k \le 24,   \sum_k m_k = d-2.\}
\end{equation*}

To prove that $r^{d}q_d$ is a polynomial, it suffices to prove that each $ r^d \prod_{k=1}^N f_{m_k/2,0}$ is a polynomial when $N\ge1, m_k\le 24, \sum_k m_k = d-2$.

By Theorem \ref{t:planar graph expand}, each $f_{m_k/2,0}$ is a linear combination of $\sin(l \theta),\cos(l \theta)$ where $\omega=(\cos \theta,\sin \theta)$, with $l\in \cI(m_k)$ with $\cI$ defined as \eqref{e:plane cI}. Combined with Lemma \ref{l:l=m mod 2}, each term in $r^d q_d(\omega)$ is of the following form
\begin{equation*}
    r^{d} \prod_j \cos(l_j \theta) \sin(\tilde{l}_j \theta), \quad \sum_j (|l_j|+|\tilde{l}_j|) \le d-2 \text{ and } d \equiv \sum_j (l_j+\tilde{l}_j) \text{ mod } 2.
\end{equation*}
This implies that $r^dq_d(\omega)$ is a polynomial in $\RR^2$ by converting the polar coordinates back to Euclidean coordinates. 

Therefore, by \eqref{e:inverse U expand}, we have
\begin{equation*}
    T - U(x) = \frac{|x|^2}{2} + \sum_{d=4}^{26} P_d(x) + P_{25}^{\log}(x) \log|x| + \cR(x)
\end{equation*}
with $\cR(r \omega) \in O_*(r^{27-\varepsilon})$ for any $\varepsilon>0$. Writing $x=(x_1,x_2)=r\omega$, we use
\begin{equation}\label{e:O* to Euclidean}
    \partial_{x_1} = \frac{1}{r}( \cos \theta r \partial_r - \sin \theta \partial_{\theta}), \quad  \partial_{x_2} = \frac{1}{r}( \sin \theta r \partial_r + \cos \theta \partial_{\theta}).
\end{equation}
Repeated differentiation implies that $|D^{\beta}\cR (x)| \le C_{\beta,\varepsilon} |x|^{27-\varepsilon-|\beta|}$ for any $\beta$ and $\varepsilon>0$. Hence the proof of Theorem \ref{t:arrival higher-order expand} is completed.
\end{proof}

A direct consequence of Theorem \ref{t:arrival higher-order expand} is that $U$ is $C^{24,\alpha} $ for any $0<\alpha<1$.

\begin{corollary}\label{c:C24 reg}
    Let $U$ be an arrival time function in some convex domain $\Omega \subset \RR^2$. Then $U \in C^{24,\alpha}$ for any $0<\alpha<1$. More precisely, for every compact $K \subset \Omega$ there exists some constant $C_K$ such that for any $x,y\in K$ and $x\neq y$ 
    \begin{equation}
        |D^{24}U(x) - D^{24}U(y) | \le C_K|x-y|\cdot (1+ \big|\log|x-y| \big|).
    \end{equation}
    Moreover, if $P_{25}^{\log} \equiv 0$ in the expansion of $U$ \eqref{e:U plane expand}, then we have $U \in C^{26,\alpha}$ for any $0<\alpha<1$.
\end{corollary}

\begin{proof}
    Take twenty-fifth derivatives in \eqref{e:U plane expand} and then integrating gives the estimates.
\end{proof}

Combining Corollary \ref{c:C24 reg} with the non-$C^{25}$ examples of Huang-Zhu \cite{huangzhu} gives the sharpness of the
planar regularity.

\begin{corollary}\label{c:not C24,1}
    There exists a planar arrival time function $U$ defined in $\Omega\subset \RR^2$ that is not $C^{24,1}$. 
\end{corollary}

\begin{proof}
    By \cite[Theorem 1.3]{huangzhu}, there exists some planar arrival time function $U$ that is not $C^{25}$. By Corollary \ref{c:C24 reg}, we know that $P^{\log}_{25} \not\equiv 0$. Choose $\omega$ with $P_{25}^{\log}(\omega)\neq 0$. Along $x= t \omega$, the twenty-fourth directional derivative of the logarithmic term is
    $25! P_{25}^{\log}(\omega) t \log t + O(t)$ as $t\downarrow0$, whereas the other terms have a bounded Lipschitz quotient.  Then the expansion \eqref{e:U plane expand} implies directly that $U$ is not $C^{24,1}$. This completes the proof. 
\end{proof}

\begin{proof}[Proof of Theorem \ref{t:intro planar}]
    This is a direct consequence of Corollaries \ref{c:C24 reg} and \ref{c:not C24,1}.
\end{proof}

\subsection{Optimal regularity in higher dimensions}

The derivative estimates from Section \ref{s:fine asymptotic expansion}
also gives the optimal regularity in dimensions $n\ge2$, thus proving Theorem \ref{t:intro higher}. We will prove that the arrival time function is $C^{2,\frac{2}{n}}$ in $\Omega \subset \RR^{n+1}$ for any $n\ge 2$. Combining with the work by \v{S}e\v{s}um \cite{se} proves that this regularity is sharp (see also \cite{str}). 

\begin{theorem}\label{t:higher dim regularity}
    Let $n\ge 2 $  and let $U$ be the arrival time of a smooth closed convex flow $M_t^n \subset \RR^{n+1}$ shrinking to $(0,T)$. Then there exist a homogeneous harmonic quadratic polynomial $H_2$ and a remainder term $\cR$ such that, for every $0<\varepsilon <1$,\begin{equation}\label{e:higher dim arrival expand}
 U(x) = T-\frac{|x|^2}{2n}
   + |x|^{2/n} H_2(x) +\cR(x),
 \qquad
 |D^\beta \cR(x)|
 \le C_{\beta, \varepsilon} |x|^{2+(3-\varepsilon)/n-|\beta|}
\end{equation}
for every multi-index $\beta$ and small $|x|$. In particular,
$U\in C^{2,2/n}$ for $n\ge 2$. 
\end{theorem}

\begin{proof}
    Let $v$ be the graph function of the flow on $\mathbb{S}^{n}(\sqrt{2n})$. By Corollary \ref{Cor;asymptotic;3} with $m=2$, for all $a,\ell\geq 0$ we have
    \begin{equation*}
	\Big\Vert \partial_{\tau}^a \Big(v - e^{-\tau/n}F_{1/n} \Big)\Big\Vert_{C^\ell(\Sigma)} = O(e^{-\frac{3-\varepsilon}{2n}\tau}),
    \end{equation*}
where $F_{1/n} \in E_2$. Write $x=r\omega$ and $s = \sqrt{T -t }= e^{-\tau/2}$. By \eqref{e:x and v, s} and \eqref{e:partial s to tau} we have
\begin{equation}\label{e:high dim r 1}
    r= s \big(\sqrt{2n} + v(\sqrt{2n} \omega, -2\log s) \big) = \sqrt{2n} \cdot s + F_{1/n}(\sqrt{2n} \omega) \cdot s^{1+2/n} + O_*(s^{1+\frac{3-\varepsilon}{n}}).
\end{equation}

By Lemma \ref{l:inverse formula}, we have $s = r/\sqrt{2n} + O_*(r^{1+2/n})$. Substituting this into \eqref{e:high dim r 1} we obtain
\begin{align*}
    s &= \frac{r}{\sqrt{2n}} - \frac{F_{1/n}(\sqrt{2n}\omega)}{\sqrt{2n}} \big( \frac{r}{\sqrt{2n}} + O_*(r^{1+\frac{2}{n}}) \big)^{1+\frac{2}{n}} + O_*(r^{1+\frac{3-\varepsilon}{n}}) \\
    &= \frac{r}{\sqrt{2n}} - \frac{F_{1/n}(\sqrt{2n} \omega)}{\sqrt{2n}} \big( \frac{r}{\sqrt{2n}} \big)^{1+\frac{2}{n}}  + O_*(r^{1+\frac{3-\varepsilon}{n}}).
\end{align*}

Since $T - U(x) = s^2$, we have
\begin{equation*}
    T - U(x) = \frac{|x|^2}{2n} - \frac{2 F_{1/n}(\sqrt{2n}\omega)}{\sqrt{2n}} (\frac{|x|}{\sqrt{2n}})^{2+\frac{2}{n}} + O_*(|x|^{2+\frac{3-\varepsilon}{n}}).
\end{equation*}

Since $F_{1/n} \in E_2$, $|x|^2F_{1/n}(\sqrt{2n}\omega)$ is a homogeneous harmonic polynomial in $\RR^{n+1}$. The derivative estimates then follow from $\cR(r\omega) \in O_*(r^{2+\frac{3-\varepsilon}{n}})$. Similar to \eqref{e:O* to Euclidean}, we could transfer the derivative estimates in polar coordinates back to Euclidean coordinates to obtain that $|D^\beta \cR(x)| \le C_{\beta, \varepsilon} |x|^{2+(3-\varepsilon)/n-|\beta|}$. This completes the proof.
\end{proof}

By \cite{se}, there exists some arrival time function in a convex $\Omega \subset \RR^{n+1}$ for which $H_2 \not\equiv 0$ in the expansion \eqref{e:higher dim arrival expand}. This implies that the $C^{2,2/n}$ regularity is sharp. 

Finally we remark that for $n\ge 2$ higher-order terms in the graph expansion of Theorem \ref{Thm;asymptotic;2} and Corollary \ref{Cor;asymptotic;3} can likewise be transferred by finite inversion of \eqref{e:high dim r 1}, giving powers of $|x|$, possible logarithms and remainders with derivative estimates. The leading term expansion suffices for Theorem \ref{t:higher dim regularity}. In dimension $n=1$, this leading term is a polynomial, so higher-order terms are needed to detect the regularity obstruction.

\end{document}